\documentclass[11pt]{article}
\makeatletter
\let\@fnsymbol\@arabic
\makeatother

\usepackage{amsmath}
\usepackage{amssymb}
\usepackage{amsfonts}
\usepackage{amsthm}
\usepackage{hyperref}
\usepackage{circledsteps}

\usepackage{psfrag}
\usepackage{graphicx}
\usepackage{url}
\usepackage{color}

\newcommand{\scrs}{\scriptstyle}

\newcommand{\eqn}[2]{\begin{equation}\label{#1}#2\end{equation}}
\newcommand{\eqnst}[1]{\begin{equation*}#1\end{equation*}}
\newcommand{\eqnspl}[2]{\begin{equation}\begin{split}\label{#1}%
   #2\end{split}\end{equation}}
\newcommand{\eqnsplst}[1]{\begin{equation*}\begin{split}%
   #1\end{split}\end{equation*}}

\def\es{\emptyset}
\def\eps{\varepsilon}
\def\bone{\mathbf{1}}

\def\tp{\widetilde{p}}
\def\bp{\mathbf{p}}
\def\p{\mathbf{p}}
\def\tbp{\mathbf{\widetilde{p}}}
\def\bP{\mathbf{P}}
\def\bq{\mathbf{q}}
\def\bbZ{\mathbb{Z}}
\newcommand{\Z}{{\mathbb Z}}
\def\Zd{{\bbZ^d}}
\def\T{\mathcal{T}}
\def\cT{\mathcal{T}}
\def\Reff{R_{\mathrm{eff}}}
\def\E{\mathbf{E}}
\def\conn{\leftrightarrow}
\def\cA{\mathcal{A}}
\def\cB{\mathcal{B}}
\def\cE{\mathcal{E}}
\def\cF{\mathcal{F}}
\def\cG{\mathcal{G}}
\def\cI{\mathcal{I}}
\newcommand{\I}{{\mathcal I}}
\def\tI{\tilde{I}}
\def\hI{\widehat{I}}
\def\cIW{\mathcal{IW}}
\def\cGgood{\mathcal{G}_{\mathrm{good}}}
\def\Ggood{\cGgood}
\newcommand{\Gtree}{\cG_{\mathrm{tree}}}
\def\cL{\mathcal{L}}
\def\cN{\mathcal{N}}
\def\tcN{\widetilde{\cN}}
\def\tr{\widetilde{r}}
\def\cR{\mathcal{R}}
\newcommand{\TS}{{\mathcal{TS}}}

\def\cU{\mathcal{U}}
\def\V{\mathcal{V}}
\def\cV{\mathcal{V}}

\def\tw{\widetilde{w}}
\def\cX{\mathcal{X}}
\newcommand{\XX}{{\mathcal X}}

\def\cY{\mathcal{Y}}
\newcommand{\YY}{{\mathcal Y}}
\def\ty{\widetilde{y}}
\def\tY{\widetilde{Y}}
\def\tz{\widetilde{z}}
\def\tZ{\widetilde{Z}}

\newcommand{\prob}{\mbox{\bf P}}
\newcommand{\lra}{\leftrightarrow}

\newcommand{\ch}[1]{#1}
\newcommand{\chnew}[1]{#1}
\newcommand{\chnnew}[1]{#1}

\theoremstyle{plain}
\newtheorem{theorem}{Theorem}[section]
\newtheorem{proposition}[theorem]{Proposition}
\newtheorem{lemma}[theorem]{Lemma}

\theoremstyle{definition}
\newtheorem{definition}[theorem]{Definition}

\numberwithin{equation}{section}

\title{Logarithmic correction to resistance}

\author{Antal A.~J\'arai\thanks{Department of Mathematical Sciences, 
University of Bath, Claverton Down, Bath, BA2 7AY, United Kingdom.
E-mail: {\tt A.Jarai@bath.ac.uk}} \and 
Dante Mata L\'opez\thanks{Department of Probability and Statistics, 
Centro de Investigaci\'on en Matem\'aticas A.C. Calle Jalisco S/N, C.P. 36240, 
Guanajuato, Mexico. E-mail: {\tt dante.mata@cimat.mx}}}

\begin{document}

\maketitle

\begin{abstract}
We study the trace of the incipient infinite oriented branching random walk 
in $\Zd \times \bbZ_+$ when the dimension is $d = 6$. Under suitable
moment assumptions, we show that 
the electrical resistance between the root and level $n$
is $O(n \log^{-\xi}n )$ for a $\xi > 0$ that does not depend 
on details of the model.\\
\emph{Key-words:} electrical resistance; branching random walk; anomalous diffusion\\
\emph{MSC 2020:} 60K50 (Primary) 60K35, 82C41, 31C20, 60J80 (Secondary)
\end{abstract}

\section{Introduction}

\chnew{
Consider a critical branching random walk in $\Zd$ conditioned to survive
forever, starting with a single individual at the origin $o$. 
The space-time points visited by this process can be turned 
into a multi-graph, by placing an edge between $(x,n)$ and $(y,n+1)$, 
whenever a particle located at $x$ at time $n$ produces an offspring 
located at $y$ at time $n+1$. We call this multi-graph the \emph{trace}. 
We regard the trace as an electrical network, where each edge has unit 
conductance. Let $R(n)$ denote the expected resistance between $(o,0)$ 
and level $n$ of the trace.

Barlow et al.~\cite[Example 1.8(iii)]{BJKS08} showed that when $d > 6$, 
one has $R(n) \asymp n$ 
%(they also studied the much more difficult model of oriented percolation). 
Answering a question of \cite{BJKS08}, 
J\'arai and Nachmias \cite{JN14} showed that for $d \le 5$ one has
$R(n) = O (n^{1-\alpha})$ for a universal constant $\alpha>0$, under suitable
moment assumptions. In the present paper we show, under the same 
moment assumptions as in \cite{JN14}, that in $d = 6$ dimensions $R(n)$
is sub-linear by at least a logarithmic factor.
}

%Consider a critical Galton-Watson tree $\cT$ conditioned to survive 
%forever, and a random walk map $\Phi : \cT \to \Zd \times \bbZ_+$ defined
%as follows. We initialize $\Phi$ by requiring that the root $\rho$ of
%$\cT$ is mapped to $(o,0)$. Then, recursively, if 
%$\{ U,V \}$ is an edge of $\cT$ between generations $n$ and $n+1$, 
%such that $\Phi(U) = (x,n)$ has already been defined, we set 
%$\Phi(V) = (y,n+1)$ with the displacement $y-x$ chosen according to 
%a symmetric random walk distribution, independently between different edges. 
%By the \emph{trace} of the branching random walk we mean the multigraph with 
%vertex set $\Phi(\cT)$, and edge set consisting of $\{ \Phi(U), \Phi(V) \}$
%for every edge $\{ U, V \}$ of $\cT$. We regard the trace as an electrical
%network, where each edge has unit conductance. Let $R(n)$ denote the
%expected resistance between $(o,0)$ and level $n$ of the trace. 

\begin{theorem}
\label{thm:expected}
Consider the trace of a branching random walk in dimension $d = 6$ 
with progeny distribution that is critical, has positive variance and finite 
third moment, conditioned to survive forever.
Assume that the random walk steps are symmetric, non-degenerate and 
have exponential tails. Then there exists a universal $\xi > 0$ such that 
\eqnst
{ R(n) 
  = O \left( n \, \log^{-\xi} n \right). }
\end{theorem}

\chnew{
Let us explain some background and our motivation. Given an infinite
graph $G$, suitable bounds on the volume growth of $G$ and electrical 
resistances in $G$ provide quantitative information on the behaviour of the 
simple random walk on $G$, such as bounds on exit times from balls, and the
heat kernel; see e.g.~\cite{BJKS08}, \cite{KM08}. One can distinguish 
two regimes. 

On the one hand, several examples are known where the behaviour of
the walk is characterised by the so-called Alexander-Orbach (AO) exponents \cite{AO82}.
See \cite{K86} for the case of a critical branching tree conditioned to survive, 
and \cite{BK06} for more detailed estimates in the case of percolation on 
regular trees. The same exponents were also shown for: the incipient infinite
cluster of oriented (spread-out) percolation in dimensions $d+1 > 6+1$ \cite{BJKS08}; 
for the oriented critical branching random in dimension $d+1 > 6+1$ 
\cite[Example 1.8(iii)]{BJKS08} ;
for unoriented percolation under the triangle condition \cite{KN09}, \cite{HvHH14}; 
and for the uniform spanning forest in $d > 4$ and on non-amenable graphs \cite{H20}.
In all of these examples, the underlying graph is similar, in a quantifiable sense,
to a critical branching process, and the scaling limit of the walk is conjectured,
and in some cases rigorously known, to be Brownian motion on the 
continuum random tree, see e.g.~\cite{C08}, \cite{BCF19}.

On the other hand, there are `low-dimensional' examples where the random walk 
exponents are different (either conjecturally, or rigorously). 
See \cite{K86}, \cite{GL20} for the incipient infinite percolation cluster in 2D; 
see \cite{BCK17} for the uniform spanning tree in 2D and 
\cite{ACHS20} for the uniform spanning tree in 3D. Despite this 
progress, it is a major challenge, for example, to establish the scaling of 
resistances for 2D critical percolation. This is due to lack of a clear connection 
between resistance and conformally invariant quantities. 

A consequence of \cite{JN14} is that for the trace of oriented critical branching random walk, 
the AO exponents cannot hold in any dimension $d+1 \le 5+1$. This results from 
the fact that the resistance does not scale linearly: $R(n) = O( n^{1-\alpha} )$ 
for some $\alpha > 0$. Although branching random walk is one of the simplest 
statistical physics models, determining the exact behaviour of the resistance 
for $d \le 6$ is already very challenging in this case. To the best of our knowledge, 
no polynomial lower bound on $R(n)$ is known in $3 \le d \le 5$, where such is expected.
\chnnew{We note that since there are $O(k)$ edges connecting generations $k$ and $k+1$, a
Nash-Williams bound \cite[Section 2.5]{LP16} yields that $R(n) = \Omega(\log n)$ for all $d \ge 1$. 
We expect this bound to be sharp when $d = 1$, and possibly also in $d = 2$.}

In this paper we consider the trace of oriented branching 
random walk when $d = 6$, which is the conjectured critical dimension. 
By analogy with other statistical physics models, one expects a 
logarithmic correction: $R(n) \asymp n \log^{-\xi'} n$ for some $\xi' > 0$ 
\chnnew{that does not depend on the offspring or random walk distribution (as long as certain 
moment assumptions are satisfied).}  
Theorem \ref{thm:expected} provides an upper bound of this form. We expect that 
with additional work one can prove a lower bound of the same form (with an 
exponent $\xi'' < \infty$), and we outline a possible strategy for this in 
\chnnew{Section \ref{ssec:strategy}}. We also explain there why this is easier 
than proving meaningful lower bounds in $3 \le d \le 5$.

%There are several examples of anomalous diffusion on fractal graphs that are
%rigorously known to be described by the so-called Alexander-Orbach exponents \cite{AO82};
%see \cite{K86}, \cite{BK06}, \cite{BJKS08}, \cite{KN09}, \cite{H19}.
%A consequence of \cite{JN14} is that for oriented branching random walk in dimensions
%$d \le 5$ these exponents cannot hold.
%Our results, together with \cite{BJKS08} and \cite{JN14}, lend support to the
%claim that for oriented branching random walk, the critical dimension to
%observe the mean-field exponents is $d_c = 6$. 

We follow a very similar setup to that of \cite{JN14}, in that we establish our
upper bound by showing that sufficiently many intersections are present in the trace
to reduce the resistance from $O(n)$ to $O(n \log^{-\xi}n)$. The main difference is
that in $d=6$ the intersections are more sparse than in $d \le 5$. In particular,
on each scale, there is only a logarithmically small probability to find intersecting
paths on that scale. Establishing this intersection estimate is more delicate compared 
to its analogue in \cite{JN14}, and some of the other estimates also need improvement.
\chnnew{We note that the definition of the trace we use requires parallel edges. This is 
necessary for applications of the parallel law (see \eqref{e:parallel} below). 
Should one replace any such parallel edges by a single edge, we expect the same results 
to hold, however, a version of \eqref{e:parallel} with weaker assumptions
would be needed to prove this.}
}

%and hence the logarithmic correction to linear scaling.

%While the exponent $\xi$ we obtain is presumably far from optimal, 
%based on estimates we prove in this paper, we expect that also $R(n) \ge c n \log^{-\xi'}n$
%for some $\xi' > \xi$, and \ch{that in fact $R(n) \asymp n \log^{-\xi''} n$ for some 
%$\xi'' > 0$.} However, we do not pursue a lower bound here.

In order to facilitate the import of the setup from \cite{JN14}, we use the following 
convention: all notation that has the same meaning as in \cite{JN14} is identical in the
present paper, and notation that has closely related meaning is denoted by a prime.
For clarity of the proofs, we found it necessary to spell out even smaller changes compared to
\cite{JN14}. However, we do take some arguments without change from \cite{JN14}, and
hence familiarity with that paper is essential to understand our arguments.

\chnew{
\subsection{Assumptions}

Let $p(k)$, $k \ge 0$ be a progeny distribution that satisfies:\\
(i) $\sum_{k} k p(k) = 1$;\\
(ii) $\sum_{k} k (k-1) p(k) = \sigma^2 \in (0,\infty)$;\\
(iii) $\sum_k k^3 p(k) \le C_3 < \infty$.\\
The \emph{incipient infinite branching process} is obtained by conditioning on
survival up to time $n$ and taking the weak limit as $n \to \infty$.
The limiting object admits the following alternative construction
\cite{K86}, \cite{LPP95}. Consider an infinite path $(V_0, V_1, \dots)$, and attach
to each $V_i$, independently, a branching tree that in its first generation follows 
the size-biased distribution:
\eqnst
{ \tp(k) 
  = (k+1) p(k+1), \quad k \ge 0, }
and follows $p$ afterwards.

Let $\bp^1(x,y)$ be a one-step random walk transition probability in $\Zd$ that satisfies:\\
(i) $\sum_{x \in \Zd} e^{b|x|} \bp^1(o,x) < \infty$ for some $b > 0$; \\
(ii) $\{ x \in \Zd : \bp^1(o,x) > 0 \}$ generates $\Zd$ as a group;\\
(iii) $\bp^1(x,y) = \bp^1(y,x)$.\\
We will denote by $\bp^n(x,y)$ the $n$-step transition probabilities.
The \emph{incipient infinite branching random walk} is obtained by first drawing 
a sample $\cT$ of the incipient infinite branching process, and then applying
a random walk map $\Phi : \cT \to \Zd \times \bbZ_+$ defined
as follows. We initialize $\Phi$ by requiring that the root $\rho$ of
$\cT$ is mapped to $(o,0)$. Then, recursively, if $\{ U,V \}$ is an edge of 
$\cT$ between generations $n$ and $n+1$, such that $\Phi(U) = (x,n)$ has already 
been defined, we set $\Phi(V) = (y,n+1)$ with the displacement $y-x$ chosen 
according to $\bp^1(x,y)$, independently between different edges. 
By the \emph{trace} of the branching random walk we mean the multi-graph with 
vertex set $\Phi(\cT)$, and edge set consisting of $\{ \Phi(U), \Phi(V) \}$
for every edge $\{ U, V \}$ of $\cT$.
}

\subsection{Electrical resistance}

For background on electrical resistance, see \cite{LP16}. We denote by 
$\Reff(x \conn y)$ the effective resistance between vertices $x$ and $y$.
We will frequently use the triangle inequality:
\eqn{e:triangle}
{ \Reff (x \conn z)
  \le \Reff (x \conn y) + \Reff(y \conn z). }
We will also use the following \emph{parallel law}: if $G_1 = (V,E_1)$ and
$G_2 = (V,E_2)$ are graphs on the same vertex set but with disjoint 
edge sets, and $G = (V,E_1 \cup E_2)$, and if $R_1 = \Reff(x \stackrel{G_1}{\conn} y)$,
$R_2 = \Reff(x \stackrel{G_2}{\conn} y)$, then  
\eqn{e:parallel}
{ \Reff(x \stackrel{G}{\conn} y)
  \le \left( \frac{1}{R_1} + \frac{1}{R_2} \right)^{-1}
  \le \frac{1}{4} (R_1 + R_2), }
where the second inequality uses that the harmonic mean is at most the
arithmetic mean.

\subsection{A finite approximation}
\label{ssec:finite-approx}

Given $n \ge 1$ and $m \ge 2n$, we define the random tree 
$\cT_{n,m}$ as in \cite{JN14}:\\
(i) consider a backbone $V_0, \dots, V_n$ with marked root $\rho = V_0$;\\
(ii) attach to each $V_i$ a critical tree that has distribution $\tp$ in 
the first step, and $p$ afterwards, and is conditioned to die out by time
$m-i$.\\
%(Note: none of the trees thus reach height $m$.)

We define $\gamma(n,x)$ as in \cite{JN14}:
\eqnst
{ \gamma(n,x)
  = \sup_{m \ge 2n} \E_{\cT_{n,m}} \Big[ \Reff \left( (o,0) \conn \Phi(V_n) \right) \Big| 
    \Phi(V_n) = (x,n) \Big], \quad x \in \Zd. }
\ch{The following theorem is our main technical result, and is an analogue of 
\cite[Theorem 1.2]{JN14}. It is expressed in terms of the norm:} 
%Our main technical result is an estimate on this quantity, and is 
%conveniently expressed in terms of the norm:
\eqnst
{ \| x \| 
  := \sqrt{\frac{1}{d} \sum_{i,j = 1}^d x_i Q^{-1}_{ij} x_j}, }
where $Q_{ij} = \sum_{x \in \Zd} x_i x_j \bp^1(o,x)$ is the covariance 
matrix of the random walk step distribution.

\begin{theorem}
\label{thm:main}
Assume $d=6$. There exists a universal constant $\xi \in (0,1/2)$ and 
$A = A(\sigma^2, C_3, \bp^1) < \infty$ such that for all $n \ge 2$ we have
\eqnst
{ \gamma(n,x)
  \le \begin{cases}
      A n (\log n)^{-\xi} & \text{when $\| x \| \le \sqrt{n}$;} \\
      A n (\log n)^{-\xi} \, 
          \left( 1 -  \frac{\log \left( \|x\|^2 / n \right)}{\log n} \right)^{-\xi} 
          & \text{when $\sqrt{n} < \|x\| \le n/2$;} \\
      A n & \text{when $\|x\| >  n/2$.}
      \end{cases} }
\end{theorem}

\begin{proof}[Proof of Theorem \ref{thm:expected} assuming Theorem \ref{thm:main}.]
As in \cite{JN14}, we have 
\eqnsplst
{ R(n) 
  &\le \lim_{m \to \infty} \E_{\cT_{n,m}} \Big[ \Reff \left( (o,0) \conn \Phi(V_n) \right) \Big] \\
  &\le \sup_{m \ge 2n} \E_{\cT_{n,m}} \Big[ \Reff \left( (o,0) \conn \Phi(V_n) \right) \Big] \\
  &= \sum_{x \in \Zd} \bp^n(o,x) \, \gamma(n,x). }
From the bound in Theorem \ref{thm:main} we have
\eqnsplst
{ R(n) 
  &\le A n (\log n)^{-\xi} \Bigg[  \sum_{x : \|x\| \le \sqrt{n}} \bp^n(o,x)
      + \sum_{x : \sqrt{n} < \|x\| \le n^{3/4}} \bp^n(o,x) \, C \\
  &\qquad\quad + \sum_{x : n^{3/4} < \|x\| \le n/2} \bp^n(o,x) \, C \, (\log n)^\xi
      + \sum_{x : n/2 < \|x\|} \bp^n(o,x) \, (\log n)^\xi \Bigg]. }
The first two sums are bounded by $1$ and $C$, respectively.
In the third and fourth sums we use that 
\eqnst
{ (\log n)^\xi
   \le C \, n^{1/2}
   \le C \, \frac{\|x\|^2}{n}, \quad \text{when $\|x\| > n^{3/4}$.} }
This gives the upper bound, \chnnew{since}
\eqnst
{ C \sum_x \bp^n(o,x) \, \frac{\|x\|^2}{n}
  = C. }
\end{proof}

\chnew{
\subsection{Strategy for a lower bound}
\label{ssec:strategy}

Let us now present a possible strategy for a lower bound on $R(n)$ in $d = 6$.
Consider independent copies $\cT_1$ and $\cT_2$ of $\cT_{n,2n}$,
and random walk mappings $\Phi_1$ and $\Phi_2$ initialized by $\Phi_1(\rho_1) = (o,0)$
and $\Phi_2(\rho_2) = (x,0)$, where $\| x \| \asymp \sqrt{n}$. The first and second
moments of the number of intersections between $\Phi_1(\cT_1)$ and $\Phi_2(\cT_2)$
are $\asymp 1$ and $\asymp \log n$, respectively, in $d = 6$, as we will see.
A key estimate we prove in this paper is the following (presented somewhat informally 
at this stage):
\eqn{e:intersection-basic}
{ \bP ( \# \text{intersections of $\Phi_1(\cT_1)$ and $\Phi_2(\cT_2)$} \ge c \log n) 
  \ge \frac{c}{\log n}. }
Suppose for what follows that one could complement this with the bound:
\eqn{e:non-intersection}
{ \bP ( \Phi_1(\cT_1) \cap \Phi_2(\cT_2) = \es ) 
  \ge 1 - \frac{C}{\log n}. }
Consider the event that in $\Phi(\cT_{2n,m})$, the edge $\{ \Phi(V_{n}), \Phi(V_{n+1}) \}$ 
is pivotal for connecting $(o,0)$ to level $2n$ of the trace. Let $\mathrm{Tr}_k$ denote 
the $\Phi$-image of all critical trees attached to the path $\{ V_{n-2^k}, \dots, V_{n-2^{k-1}} \}$,
$k = 1, 2, \dots, \log_2 n$, and let $\widetilde{\mathrm{Tr}}$ denote the $\Phi$-image
of all critical trees attached to the path $\{ V_{n+1}, \dots, V_{2n} \}$. By the 
FKG inequality, and a heuristic based on \eqref{e:non-intersection}, we get
\eqnst
{ \bP ( \text{$\{ \Phi(V_{n}), \Phi(V_{n+1}) \}$ is pivotal} )
  \ge \prod_{k=1}^{\log_2 n} \bP ( \mathrm{Tr}_k \cap \widetilde{\mathrm{Tr}} = \es )
	\ge \prod_{k=1}^{\log_2 n} \left( 1 - \frac{C}{\log 2^k} \right)
	\ge c \, (\log n)^{-O(1)}. }
This would imply that there are at least $ n \log^{-\xi''} n$ pivotals along the backbone 
of $\cT_{2n,m}$, and hence a lower bound on the resistance.
For the same ideas to be fruitful in $3 \le d \le 5$, one would also need the exponent
to be $< 1$, not only $O(1)$.
}

\subsection{Random walk estimates}

The following proposition collects some random walk estimates we 
take without change from \cite{JN14}. Let $S(n)$, $n \ge 0$ denote
a random walk with $S(0) = o$, and step distribution $\bp^1$.

\begin{proposition}[{\cite[Proposition 1.3]{JN14}}]
\label{prop:variance}
There exists $k_1 = k_1(\bp^1), C > 0$ \ch{and $\delta_1 = \delta_1(d) > 0$} 
such that the following hold.\\
\begin{itemize}
\item[(i)] Whenever $k_1 \le k \le n$, $\| x \| \le 4n/\sqrt{k}$, we have
\eqnst
{ \E \Big[ \| S(k) \|^2 \,\Big|\, S(n) = x \Big]
  \le C k. }
\item[(ii)] Whenever $k_1 \le k \le \delta_1 n$, $\| x \| \le 4n / \sqrt{k}$, we have
\eqnst
{ \E \Big[ \| S(k) \|^2 \,\Big|\, S(n) = x,\, \| S(k) \| > \sqrt{k} \Big]
  \le C k. }
\item[(iii)] Whenever $k_1 \le k \le \delta_1 n$, $k_1 \le k' \le n-k$ and
$\| x \| \le \min \{ 4n/\sqrt{k},\, 4n/\sqrt{k'} \}$, we have
\eqnst
{ \E \Big[ \| S(k+k') - S(k) \|^2 \,\Big|\, S(n) = x,\, \| S(k) \| > \sqrt{k} \Big]
  \le C k'. }
\end{itemize}
\end{proposition}

Let us write
\eqnst
{ D 
  := \det(Q)^{1/2d}. }
The following lemma is a special case of \cite[Lemma 1.4]{JN14} (take $\beta = 0$ there).

\begin{lemma}
\label{lem:LCLT}
There exists $C = C(d)$ such that the following hold.\\
(i) There exists $n_1 = n_1(\p^1)$ such that for all $y \in \Z^d$ we have
\eqn{e:beta-LCLT-ub}
{ \prob ( S(n) = y )
  \le \frac{2 C}{D^d n^{d/2}} \, ,}
when $n \ge n_1$.\\
(ii) For any $0 < \eps < 1$ and $0 < L < \infty$ there exists
$n_2 = n_2(\p^1,\eps,L)$ such that for all $y \in \Z^d$
such that $\| y \| \le L \sqrt{n}$ we have
\eqnspl{e:beta-LCLT}
{ \prob ( S(n) = y )
  &\le \frac{C(1 + \eps)}{D^d n^{d/2}} e^{ - d \| y \|^2 / (2 n)}\, , \\
  \prob ( S(n) = y )
  &\ge \frac{C(1 - \eps)}{D^d n^{d/2}} e^{ - d \| y \|^2 / (2 n)}\, .}
when $n \ge n_2$.
\end{lemma}

Above we assumed that the random walk has period $1$. Trivial modifications 
can be made when the period is $2$, and we will not make this explicit in
our arguments.

Below we collect a few more frequently used facts. These are all
standard (see \cite[Section 1.5]{JN14} for more information).
There exists a constant $c > 0$ such that we have
\eqn{e:volume-lower-bound}
{ \sum_{x : \| x \| \le L} 1 
  \ge c D^d L^d, \quad L \ge 1. }
When $d \ge 3$, the Green function $G(x) := \sum_{n=0}^\infty \bp^n(o,x)$ satisfies
\eqn{e:Green-ub}
{ G(x) 
  \le \frac{C(d)}{D^d} \| x \|^{2-d}, \quad
      \text{when $\| x \| \ge L_1 = L_1(\bp^1)$.} }

\section{Induction Scheme}

We set up the induction scheme as in \cite{JN14}, apart from the 
definition of the event $\cB(i,c_0)$. For convenience of the reader, 
we provide Definitions \ref{udp}--\ref{intersect} below, that are from \cite{JN14}. 
Given an instance of $\cT_{n,m}$, consider a 
small $\delta > 0$, such that $\delta n$ is an integer. We write
\eqnst
{ X_i = V_{i \delta n}, \quad i = 0, 1, \dots, \lfloor \delta^{-1} \rfloor, }
and write $x_i \in \Zd$, $i = 0, 1, \dots, \lfloor \delta^{-1} \rfloor$ for the
spatial location of $X_i$, so that $\Phi(X_i) = (x_i, i \delta n)$.
Write $\cT_{n,m}(\ell)$ for the subtree of $\cT_{n,m}$ emanating from $V_\ell$ off
the backbone (including $V_\ell$). 

Fix an integer $K$ and write $n = N K \delta n + K' \delta n + \tilde{n}$, with 
$0 \le K' < K$ an integer and $\delta n \le \tilde{n} < 2 \delta n$.
The definitions to follow are illustrated in Figure \ref{fig:K-tree-good}.

\begin{definition} 
\label{udp} 
For $\ell$ satisfying $i \delta n \leq \ell < (i+1)\delta n$
we say that a backbone vertex $V_\ell$ has the {\em unique descendant property} (UDP) if
among its descendants at level $(i+1) \delta n$ in $\cT_{n,m}(\ell)$ there is a unique
one that reaches level $(i+2)\delta n$. For any other vertex $V$ of $\cT_{n,m}$ at level
$i \delta n$ we say that $V$ has UDP if among its descendants at level $(i+1)\delta n$
there is a unique one that reaches level $(i+2)\delta n$.
\end{definition}

\begin{definition}
\label{ktreegood}
Given an integer $K \geq 1$, a number $\delta>0$ such that $K \delta \le (1/2)$
and an instance of $\cT_{n,m}$ we say that a
sequence $(i, i+1, \ldots, i+K)$ of length $K+1$ is $K$-{\em tree-good} if the following holds:
\begin{enumerate}
\item[(1)] There exists a unique $i \delta n \le \ell_1 < (i+1) \delta n$ such
that $\cT_{n,m}(\ell_1)$ reaches height $(i+2)\delta n$. Moreover,
this unique $\ell_1$ satisfies $(i+1/4) \delta n \leq \ell_1 \leq (i+3/4)\delta n$.
\item[(2)] $V_{\ell_1}$ has UDP.
We call the unique descendant $\YY_{i+1}$.
For all $i'$ satisfying $i+2 \leq i' \leq i+K$ we inductively define the vertices $\YY_{i'}$ of
$\cT_{n,m}(\ell_1)$ as follows. We require that $\YY_{i'-1}$ has UDP
and call the unique descendant $\YY_{i'}$.
\item[(3)] There exists a unique $(i+K-1)\delta n \le \ell_2 < (i+K) \delta n$ such that
$\cT_{n,m}(\ell_2)$ reaches height $(i+K+1)\delta n$. Moreover, this unique $\ell_2$ satisfies
$(i+K-3/4) \delta n \leq \ell_2 \leq (i+K-1/4) \delta n$.
\item[(4)] $V_{\ell_2}$ has UDP, and we call the unique descendant
$\XX'_{i+K}$. The vertex $\XX'_{i+K}$ has UDP,
and we call the unique descendant $\XX'_{i+K+1}$.
Similarly, $\YY_{i+K}$ has UDP,
and we call the unique descendant $\YY_{i+K+1}$.
\end{enumerate}
\end{definition}

\psfrag{Xi}{$\scrs{X_i}$}
\psfrag{Vell1}{$\scrs{V_{\ell_1}}$}
\psfrag{Vell1+}{$\scrs{V_{\ell_1}^+}$}
\psfrag{Xi+1}{$\scrs{X_{i+1}}$}
\psfrag{Xi+2}{$\scrs{X_{i+2}}$}
\psfrag{Xi+K-1}{$\scrs{X_{i+K-1}}$}
\psfrag{Vell2}{$\scrs{V_{\ell_2}}$}
\psfrag{Vell2+}{$\scrs{V_{\ell_2}^+}$}
\psfrag{Xi+K}{$\scrs{X_{i+K}}$}
\psfrag{X'i+K}{$\scrs{\XX'_{i+K}}$}
\psfrag{X'i+K+1}{$\scrs{\XX'_{i+K+1}}$}
\psfrag{Yi+1}{$\scrs{\YY_{i+1}}$}
\psfrag{Yi+2}{$\scrs{\YY_{i+2}}$}
\psfrag{Yi+K-1}{$\scrs{\YY_{i+K-1}}$}
\psfrag{Yi+K}{$\scrs{\YY_{i+K}}$}
\psfrag{Yi+K+1}{$\scrs{\YY_{i+K+1}}$}
\psfrag{dn}{$\scrs{\delta n}$}

\psfrag{xi}{$\scrs{x_i}$}
\psfrag{vell1}{$\scrs{v_{\ell_1}}$}
\psfrag{vell1+}{$\scrs{v_{\ell_1}^+}$}
\psfrag{xi+1}{$\scrs{x_{i+1}}$}
\psfrag{xi+2}{$\scrs{x_{i+2}}$}
\psfrag{xi+K-1}{$\scrs{x_{i+K-1}}$}
\psfrag{vell2}{$\scrs{v_{\ell_2}}$}
\psfrag{vell2+}{$\scrs{v_{\ell_2}^+}$}
\psfrag{xi+K}{$\scrs{x_{i+K}}$}
\psfrag{x'i+K}{$\scrs{x'_{i+K}}$}
\psfrag{x'i+K+1}{$\scrs{x'_{i+K+1}}$}
\psfrag{yi+1}{$\scrs{y_{i+1}}$}
\psfrag{yi+2}{$\scrs{y_{i+2}}$}
\psfrag{yi+K-1}{$\scrs{y_{i+K-1}}$}
\psfrag{yi+K}{$\scrs{y_{i+K}}$}
\psfrag{yi+K+1}{$\scrs{y_{i+K+1}}$}
\psfrag{<=sqdn}{$\scrs{\le \sqrt{\delta n}}$}

\begin{figure}[tbph]
\begin{center}
\includegraphics[scale=0.65]{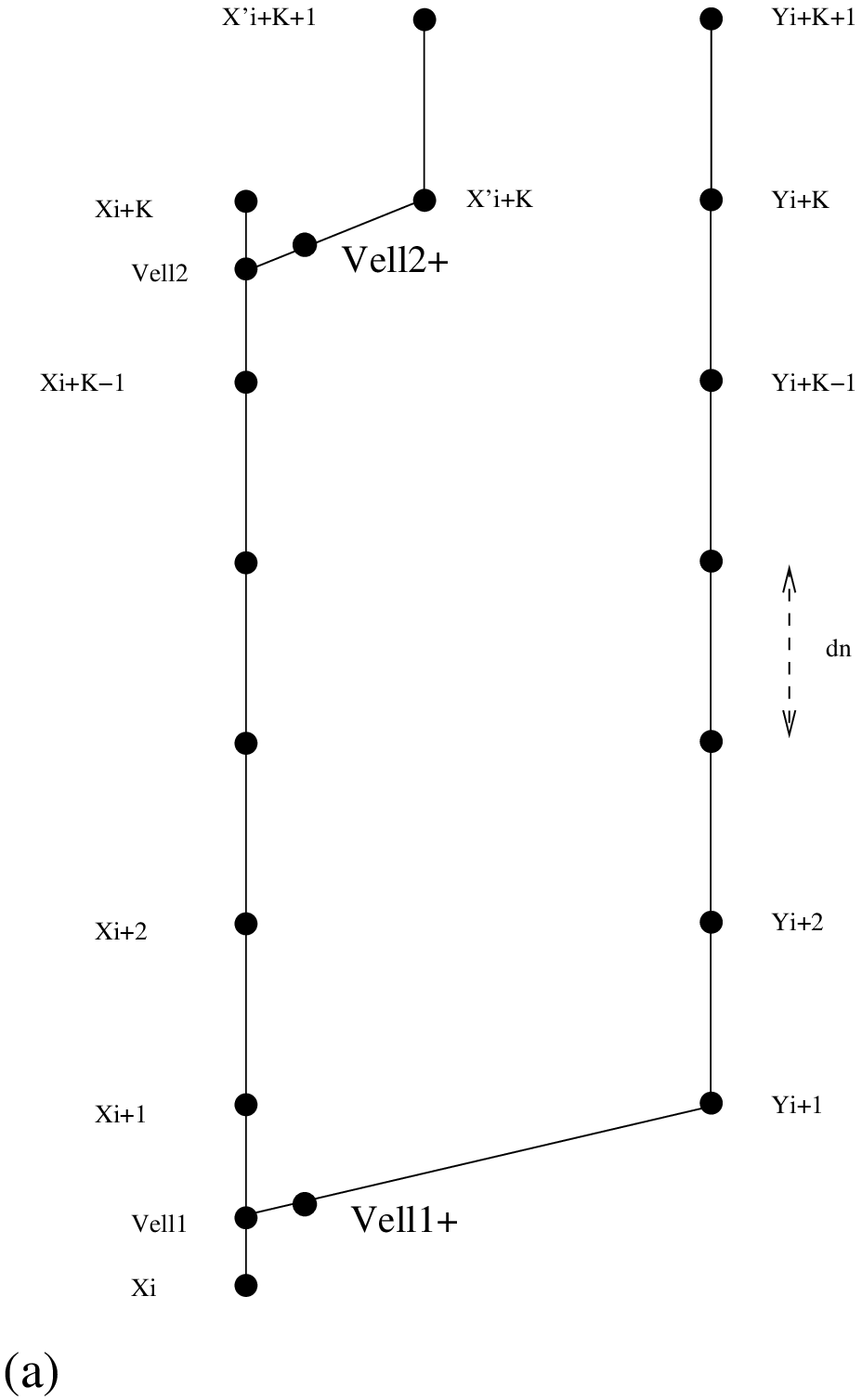}
\hskip1cm
\includegraphics[scale=0.65]{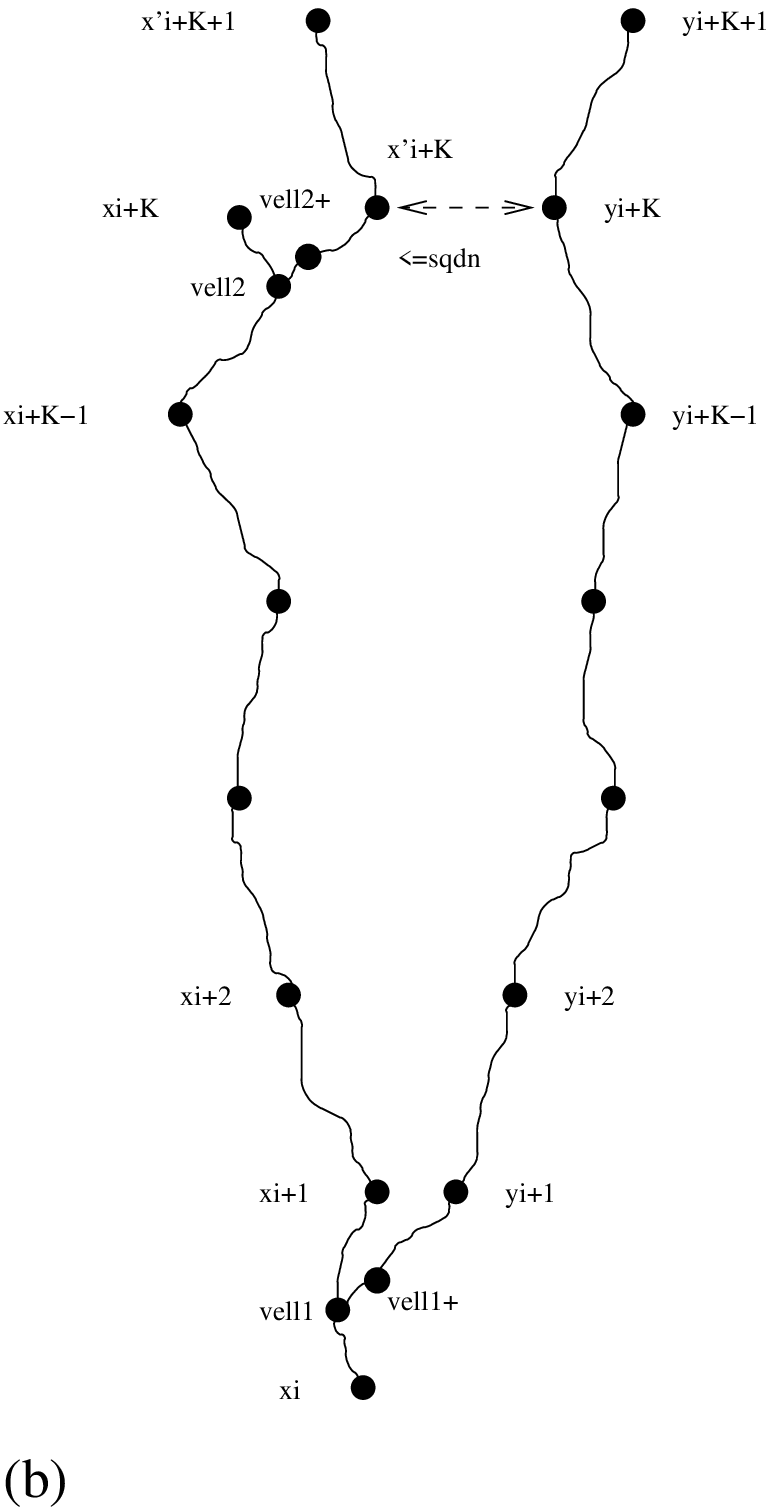}
\end{center}
\caption{(a) Illustration of $K$-tree-good.
(b) Illustration of $K$-spatially-good. All spatial distances
between consecutive vertices are at most
$\sqrt{\text{time difference}}$ and
the spatial distance between $x'_{i+K}$ and $y_{i+K}$
is at most $\sqrt{\delta n}$. 
[Figure reused by permission of Springer, from \emph{Commun.~Math.~Phys.} 
\textbf{331} (2014), 67--109 (Electrical resistance of the low-dimensional 
critical branching random walk. A.A.~J\'arai and A.~Nachmias) \CircledTop{c} (2014).]}
\label{fig:K-tree-good}
\end{figure}

Given a $K$-tree-good sequence $(i, \ldots, i+K)$
we denote by $V_{\ell_1}^+$ (respectively $V_{\ell_2}^+$)
the child of $V_{\ell_1}$ (respectively $V_{\ell_2}$)
leading to $\YY_{i+1}$ (respectively $\XX'_{i+K}$).
We further define the spatial locations $y_{i'}$ by
$\Phi(\YY_{i'}) = (y_{i'}, i' \delta n)$ for
$i+1 \leq i' \leq i+K\ch{+1}$, and
we similarly define $x'_{i+K}$, $x'_{i+K+1}$,
$v_{\ell_1}$, $v_{\ell_1}^+$, $v_{\ell_2}$, $v_{\ell_2}^+$.

We will write $U \prec W$ to denote that $W$ is a descendant of $U$, 
and write $h(U), h(W)$ for their respective heights in the tree 
(in particular, $h(W)>h(U)$).

\begin{definition}
\label{typicallyspaced}
Let $U \prec W$ be two tree vertices and let $u,w\in\Z^d$ be defined by
$\Phi(U)=(u,h(U))$ and $\Phi(W)=(w,h(W))$. We say that $U$ and $W$ are
{\em typically-spaced} if $\| w - u \| \leq \sqrt{h(W)-h(U)}$.
Denote this event by $\TS(U,W)$.
\end{definition}

\begin{definition}
\label{kspatialgood}
We say that a $K$-tree-good sequence $(i,\ldots, i+K)$ is
\emph{$K$-spatially-good} if the following holds.
\begin{enumerate}
\item[(5)] \begin{itemize}  \item $\TS(X_i, V_{\ell_1})$,
                            \item $\TS(V_{\ell_1+1}, X_{i+1})$,
                            \item For each $i+1 \leq j \leq i+K-2$ we have $\TS(X_j, X_{j+1})$,
                            \item $\TS(X_{i+K-1}, V_{\ell_2})$,
                            \item $\TS(V_{\ell_2+1}, X_{i+K})$,
            \end{itemize}
\item[(6)] \begin{itemize}  \item $\TS(V_{\ell_1}^+, \YY_{i+1})$,
                            \item For each $i+1 \leq j \leq i+K-1$ we have $\TS(\YY_j, \YY_{j+1})$,
                            \item $\TS(V_{\ell_2}^+, \XX'_{i+K})$,
                            \item $\| x'_{i+K} - y_{i+K} \| \le \sqrt{\delta n}$.
            \end{itemize}
\end{enumerate}
\end{definition}

\begin{definition} 
\label{def:kgood} 
When a sequence $(i, \ldots, i+K)$ is both $K$-tree-good
and $K$-spatially-good we say that it is $K$-{\em good}. Let $\cA(i)$ be the
event that $(i, \ldots, i+K)$ is $K$-good.
\end{definition}

Next, let $(i, \ldots, i+K)$ be a $K$-good sequence and let $U_1, U_2$  be two vertices at the same height such
that $U_1 \succ \XX'_{i+K}$ and $U_2 \succ \YY_{i+K}$. Given these, we write $Z_1$ for the highest common ancestor
of $U_1$ and $\XX'_{i+K+1}$ and $Z_2$ for the highest common ancestor of $U_2$ and $\YY_{i+K+1}$
(see Figure \ref{fig:intersect}). Further, we denote by $Z_1^+$ (respectively $Z_2^+$)
the child of $Z_1$ (respectively $Z_2$) leading to $U_1$ (respectively $U_2$).

\psfrag{Phi(X'i+K)}{$\scrs{\Phi(\XX'_{i+K})}$}
\psfrag{Phi(Yi+K)}{$\scrs{\Phi(\YY_{i+K})}$}
\psfrag{Phi(X'i+K+1)}{$\scrs{\Phi(\XX'_{i+K+1})}$}
\psfrag{Phi(Yi+K+1)}{$\scrs{\Phi(\YY_{i+K+1})}$}
\psfrag{Phi(U_1) = Phi(U_2)}{$\scrs{\Phi(U_1) = \Phi(U_2)}$}
\psfrag{Phi(Z_1)}{$\scrs{\Phi(Z_1)}$}
\psfrag{Phi(Z_2)}{$\scrs{\Phi(Z_2)}$}
\psfrag{Phi(Z+_1)}{$\scrs{\Phi(Z^+_1)}$}
\psfrag{Phi(Z+_2)}{$\scrs{\Phi(Z^+_2)}$}

\begin{figure}
\begin{center}
\includegraphics[scale=0.6]{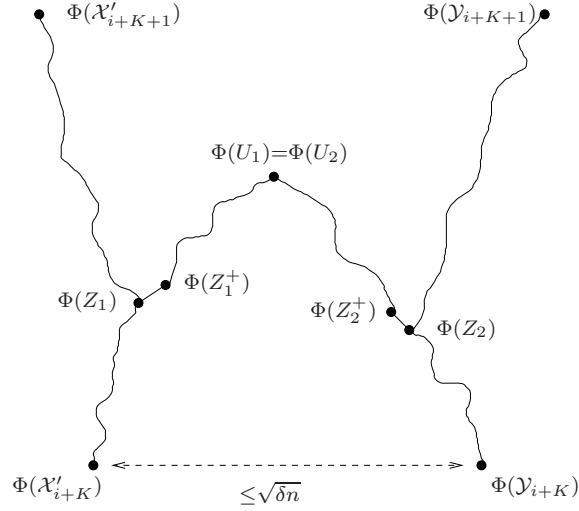}
\end{center}
\caption{The labelling of vertices in the two (potentially)
intersecting trees emanating from $\XX'_{i+K}$ and $\YY_{i+K}$.
[Figure reused by permission of Springer, from \emph{Commun.~Math.~Phys.} 
\textbf{331} (2014), 67--109 (Electrical resistance of the low-dimensional 
critical branching random walk. A.A.~J\'arai and A.~Nachmias) \CircledTop{c} (2014).]}
\label{fig:intersect}
\end{figure}

\begin{definition}
\label{intersect} 
We say that $U_1, U_2$ \emph{intersect-well} if the following conditions hold:
\begin{itemize}
\item[1.] $U_1 \succ \XX'_{i+K}$, $U_2 \succ \YY_{i+K}$,
\item[2.] $(i+K+(5/6)) \delta n \le h(U_1) = h(U_2) \le (i+K+1) \delta n$;
\item[3.] $(i+K+(1/2)) \delta n \le h(Z_1), h(Z_2) \le (i+K+(4/6)) \delta n$;
\item[4.] $\TS(\XX'_{i+K}, Z_1)$, $\TS(Z_1^+, U_1)$, $\TS(\YY_{i+K}, Z_2)$, $\TS(Z_2^+, U_2)$;
\item[5.] $\Phi(U_1) = \Phi(U_2)$.
\end{itemize}
And define the random set $\I$ by
\begin{equation}
\label{e:I-defn}
\begin{split}
  \cI  = \left\{ (U_1,U_2) : \text{$U_1$ and $U_2$ intersect-well} \right\}\, ,
\end{split}
\end{equation}
Now we define the event $\cB'(i,c'_0)$ where $c'_0>0$ is a constant
to be chosen later:
\eqnst
{ \cB'(i,c'_0) 
  = \cA(i) \cap \left\{ |\cI| \geq \frac{c'_0 \sigma^4}{D^d} \log (\delta n) \right\} \, .}
\end{definition}

\ch{The following theorem replaces \cite[Theorem 2.1]{JN14}, 
and will be proved in Section \ref{sec:intersect}.}

\begin{theorem}[Intersections exist]
\label{thm:intersect}
Assume $d=6$. There exist constants $c'_0, c'_1 > 0$ and for any $K \ge 2$ there
exists $c_2 = c_2(K) > 0$, and $n_3' = n_3'(\sigma^2, C_3, \bp^1, K)$ such that 
for any $0 < \delta < (K+4)^{-1}$, whenever $\delta n \ge n_3'$ and $x$ satisfies
$\|x\| \le \sqrt{2 n/\delta}$, we have
\eqnst
{ \bP ( \cA(i) \,|\, \Phi(V_n) = (x,n) ) \ge c_2, }
and 
\eqnst
{ \bP ( \cB'(i,c'_0) \,|\, \cA(i), \, \Phi(V_n) = (x,n) )
  \ge \frac{c'_1}{\log (\delta n)}, }
for $i = 0, K, 2K, \dots, (N-1)K$.
\end{theorem}

As in \cite{JN14}, we define 
\eqnst
{ \gamma(n)
  = \sup_{x : \|x\| \le \sqrt{n}} \gamma(n,x). }
The proof of the following theorem will be completed in Section \ref{sec:good-blocks}.

\begin{theorem}[Analysis of good blocks]
\label{thm:good-blocks}
There exists $K'_0 < \infty$ and $n'_4 = n'_4(\sigma^2, C_3, \bp^1)$ such that 
if $K \ge K'_0$ and $\delta n \ge n'_4$, we have
\eqnst
{ \E \Big[ \Reff ( \Phi(X_i) \conn \Phi(X_{i+K}) ) \Big| \cA(i),\, \cB'(i,c'_0),\, 
     \Phi(V_n) = (x,n) \Big]
  \le \frac{3 K}{4} \, \max_{1 \le k \le \delta n} \gamma(k), }
for $i = 0, K, 2K, \dots, (N-1)K$.
\end{theorem}

\section{Existence of Intersections}
\label{sec:intersect}

In this section we prove Theorem \ref{thm:intersect}. 
The statement about the probability of $\cA(i)$ is unchanged compared to
\cite[Theorem 2.1]{JN14}, and hence requires no proof. On the other hand, there are 
substantial changes to the proof of the estimate on the probability 
of $\cB'(i,c'_0)$, that we now detail. Some of the required estimates
appeared in the MSc thesis \cite{DML18}, in a slightly 
different form (with less technical restrictions on the intersections). 
Here we adapt and complete the analysis of  
\cite{DML18} in a form that suits the requirements of the present paper.

\subsection{Sufficient intersections}

We now proceed to prove the second statement of Theorem \ref{thm:intersect}.
We will need to assume that the progeny distribution is bounded by $M$,
and approximate the original progeny distribution with bounded 
distributions $p_M(k)$, in such a way that 
\eqnsplst
{ 1 
  &= \sum_{0 \le k \le M} p_M(k) 
  = \sum_{0 \le k \le M} k p_M(k) \\
  \sigma^2
  &= \lim_{M \to \infty} \sum_{0 \le k \le M} k (k-1) p_M(k) \\
  C_3 
  &\ge \sum_{0 \le k \le M} k^3 p_M(k). }
Given any $n$ and $m$ such that $m\geq 2n$ we regard the random tree $\T_{n,m}$ as a 
subtree of an infinite $M$-ary tree $T_M$ with root $\rho$ as follows: the root of 
$\T_{n,m}$ is mapped to $\rho$ and if $W$ is a vertex of $\T_{n,m}$ with $k$ children 
we map the $k$ edges randomly amongst the $\binom{M}{k}$ possible choices in $T_M$. 
Denote by $\V_n \in T_M$ the random vertex where the last backbone vertex of $\T_{n,m}$ 
was mapped to. The triple $(\T_{n,m}, \rho, \V_n)$ is a doubly rooted tree.

Let $(\cT_1,\rho_1,\cV_1)$ and $(\cT_2,\rho_2,\cV_2)$ be two independent copies of
$(\cT_{\delta n, 2 \delta n}, \rho, \cV_{\delta n})$, randomly imbedded into 
$T_M$. Let $\Phi_1$ and $\Phi_2$ be two independent random walk mappings of 
$T_M$ such that $\Phi_1(\rho_1) = \Phi(\cX'_{i+K})$ and $\Phi_2(\rho_2) = \Phi(\cY_{i+K})$.
However, for notational convenience, and without loss of generality, we assume that 
$\Phi_1(\rho_1) = (o,0)$ and $\Phi_2(\rho_2) = (x,0)$, with $\| x \| \le \sqrt{\delta n}$.
Recall that the random variable $|\cI|$ introduced in \eqref{e:I-defn} has the same distribution as
the random variable (also denote $|\cI|$ here):
\eqnst
{ |\cI|
  = \sum_{U_1 \in \cT_1,\, U_2 \in \cT_2} \bone_{\text{$(U_1,U_2)$ intersect well}}. }
Our goal in this section is to show that when $d=6$, we have
\eqn{e:I-estimate}
{ \bP \left( |\cI| \ge c'_0 \sigma^4 D^{-d} \log (\delta n) \right)
  \ge \frac{c'_1}{\log (\delta n)} }
for suitable $c'_0, c'_1 > 0$.
Note that once we prove this estimate for offspring distribution $p_M(k)$
(uniformly in $M$), we can let $M \to \infty$, and obtain the second statement of
Theorem \ref{thm:intersect} in full generality.

For technical reasons, we will also need a slight modification of $\cI$.
The difference is that we make a stronger restriction on the spatial
displacement between $\Phi(Z_1^+)$ and $\Phi(U_1)$ and between $\Phi(Z_2^+)$ and
$\Phi(U_2)$, as well as a stronger restriction on the height of the intersection:
\eqnst
{ \cI'
  = \left\{ (U_1, U_2) \in \cI : \quad \parbox{6cm}{$h(U_1) = h(U_2) \le (11/12) \delta n$, \\ 
      $\| z_1^+ - u_1 \| \le \frac{1}{2} \sqrt{h(U_1) - h(Z_1^+)}$, \\ 
      $\| z_2^+ - u_2 \| \le \frac{1}{2} \sqrt{h(U_2) - h(Z_2^+)}$} \right\}. }

The starting point for proving \eqref{e:I-estimate} is to look at the moments of $|\cI|$.
The following theorem is an analogue of \cite[Theorem 3.8]{JN14}.

\begin{theorem}
\label{thm:I-moments}
Assume that $d=6$ and $\|x\| \le \sqrt{\delta n}$. There exist constants $C' < \infty$
and $c' > 0$ and $n'_9 = n'_9(\sigma^2, C_3, \bp^1) < \infty$ such that for 
$\delta n \ge n_9$ we have
\eqn{e:I-1st-moment}
{ \E |\cI| 
  \ge \E |\cI'| 
  \ge \frac{c' \sigma^4}{D^6}, }
and
\eqn{e:I-2nd-moment}
{ \E |\cI|^2
  \le \frac{C' \sigma^8}{D^{12}} \log (\delta n). }
\end{theorem}

Since the proof of this theorem requires only minor adaptations compared 
to the proof of \cite[Theorem 3.8]{JN14}, we omit it.

Applying the Paley-Zygmund inequality to $|\cI|$ shows that 
$\bP ( |\cI| > 0 ) \ge c' / \log (\delta n)$, but comes short of proving
\eqref{e:I-estimate}. In order to prove this, we 
will consider the number of intersections conditional on
$(U_1,U_2) \in \cI'$ for fixed $U_1, U_1 \in T_M$, and show that 
this is at least of order $\log (\delta n)$ with conditional 
probability bounded away from $0$.

%We need a number of definitions to set up the required estimate.
%Define
%\eqnst
%{ q(k)
%  := \binom{M}{k}^{-1} p(k) \, .}
%
%\begin{lemma}[{\cite[Lemma 6.3]{JN14}}]
%\label{proboftree} 
%For a fixed triple $(t,\rho,V)$ where $t \subset T_M$ is a tree and $V\in T_M$ 
%at height $n$ such that $t$ does not reach level $m$ and $V$ has no children 
%in $t$, we have
%\eqn{e:T1-formula}
%{ \prob ( (\T_{n,m},\rho, \V_n) = (t, \rho, V) )
%  \propto
%    \prod_{\substack{W \in t \\ W \not= V}} q(\deg^+_t(W)) \, ,}
%where $\deg^+_t(W)$ is the number of children of $W$ in $t$.
%\end{lemma}

For the statement of the next lemma we fix
\eqnst
{ 0 \le k_1 \leq \delta n-1 \qquad\qquad
  k_1 + 1 \leq h_u \leq \delta n.
}
Given $V \in T_M$ at level $\delta n$ and $U \in T_M$ at level $h_u$ let 
$Z\in T_M$ be the highest common ancestor of $V$ and $U$ and let $Z^+$ be 
the unique child of $Z$ leading towards $U$. 
\ch{We assume that $Z$ is at height $k_1$.}
Given a tree $t \subset T_M$ 
such that $V, U\in t$ and $V$ does not have any children in $t$, we have 
a unique decomposition of $t$ into edge disjoint trees 
$(t^A, \rho, Z), (t^B, Z^+, U), t^C$ and $t^D$, see Figure \ref{fig:to-U1}. 
The doubly rooted tree $(t^A, \rho, Z)$ contains all the descendants of 
$\rho$ that are not descendants of $Z$. The doubly rooted tree 
$(t^B, Z^+, U)$ contains all the descendants of $Z^+$ that are not 
descendants of $U$. The tree $t^C$ contains all the descendants of $U$ 
and finally the tree $t^D$ contains all other edges, namely, all the 
descendants of $Z$ that are not descendants of $Z^+$ (in particular, 
the edge $Z,Z^+$ is in $t^D$).

For $W \in T_M$ let $\Theta_W$ denote the tree isomorphism that takes 
$W$ to $\rho$ and the descendants subtree of $W$ onto $T_M$.

\psfrag{deltan}{$\delta n$}
\psfrag{2deltan}{$2 \delta n$}
\psfrag{ell1}{$k_1$}
\psfrag{h_u}{$h_u$}
\psfrag{zero}{$0$}
\psfrag{Y_1}{$V$}
\psfrag{U_1}{$\!U$}
\psfrag{Z_1}{$\!\!\!Z$}
\psfrag{Z^+_1}{$Z^+$}
\psfrag{rho}{$\rho$}
\psfrag{t1}{$t^A$}
\psfrag{t2}{$t^B$}
\psfrag{t3}{$t^C$}
\psfrag{t4}{$t^D$}

\begin{figure}
\begin{center}
\includegraphics[scale=0.51]{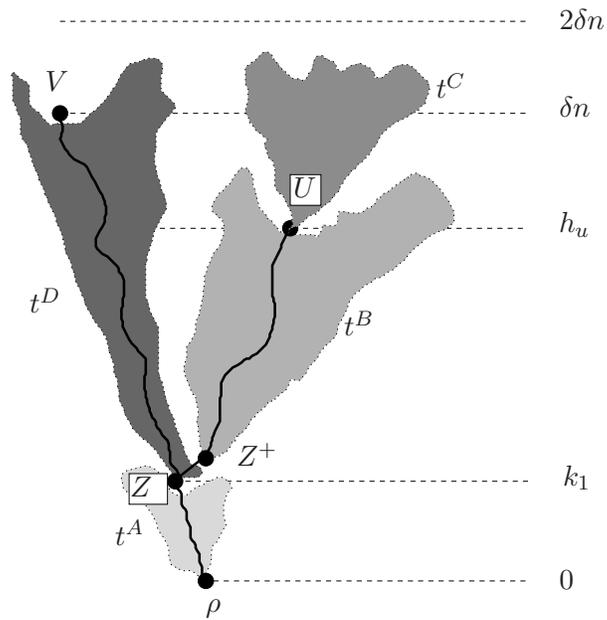}
\end{center}
\caption{Illustration of the decomposition into
edge-disjoint trees $t^A, t^B, t^C, t^D$ appearing
in Lemma \ref{lem:distrib} ($2 \delta n$ and $\delta n$
are not to scale).
[Figure reused by permission of Springer, from \emph{Commun.~Math.~Phys.} 
\textbf{331} (2014), 67--109 (Electrical resistance of the low-dimensional 
critical branching random walk. A.A.~J\'arai and A.~Nachmias) \CircledTop{c} (2014).]}
\label{fig:to-U1}
\end{figure}

The following lemma is taken without change from \cite{JN14}.

\begin{lemma}[{\cite[Lemma 6.4]{JN14}}]
\label{lem:distrib}
Let $V, U \in T_M$ be at heights $\delta n$ and $h_u$, respectively, and let 
$(\cT, \rho, \cV)$ be distributed as $(\cT_{\delta n, 2 \delta n}, \rho, \cV_{\delta n})$.
Conditionally on the event $\{ \cV = V,\, U \in \cT \}$ we have that
\eqnst
{ ( \cT^A,\, \rho,\, Z ) 
  \stackrel{\mathrm{d}}{=} ( \cT_{k_1, 2 \delta n},\, \rho,\, \cV_{k_1} ) \,|\, \cV_{k_1} = Z, }
and 
\eqnst
{ \Theta_{Z^+}( (\cT^B,\, Z^+,\, U) )
  \stackrel{\mathrm{d}}{=} ( \cT_{h_u - k_1 - 1, 2 \delta n - k_1 - 1},\, \rho,\, \cV_{h_u - k_1 - 1} )
     \,|\, \cV_{h_u - k_1 - 1} = \Theta_{Z^+}(U). }
\end{lemma}

Let us fix vertices $V_1, V_2 \in T_M$ at height $\delta n$, and
vertices $U_1, U_2 \in T_M$, both at height $h_u$, where 
$(5/6) \delta n \le h_u \le (11/12) \delta n$, such that 
$(1/2) \delta n \le h(Z_1), h(Z_2) \le (4/6) \delta n$.
Let us denote the vertices on the path in $T_M$ between $Z_1^+$ and 
$U_1$ as
\eqnst
{ Z_1^+ = W_1^{(k_1+1)}, W_1^{(k_1+2)}, \dots, W_1^{(h_u-1)}, W_1^{(h_u)} = U_1, }
and vertices on the path in $T_M$ between $Z_2^+$ and $U_2$ as 
\eqnst
{ Z_2^+ = W_2^{(k_2+1)}, W_2^{(k_2+2)}, \dots, W_2^{(h_u-1)}, W_2^{(h_u)} = U_2. }

For any choice of $(V_1,V_2,U_1,U_2)$ as above, let us define the event
\eqnst
{ \cIW'
  = \cIW'(V_1,V_2,U_1,U_2)
  = \{ \cV_1 = V_1,\, \cV_2 = V_2 \} \cap \{ (U_1,U_2) \in \cI' \}. }
We will also use the shorthand
\eqnst
{ \bone_{\cIW'}
  = \bone_{\cIW'(V_1,V_2,U_1,U_2)}. }
For the next definition, recall the constants $n_1$ and $n_2$ of Lemma \ref{lem:LCLT},
and put 
\eqnst
{ \ch{n^* 
  = \max \{ n_1, n_2(\bp^1, \eps = 1/2, L = 1) \}.} } 
When the event $\cIW'$ occurs, we say that a pair of vertices
$(Y_1,Y_2)$ is a \emph{good extra intersection}, if the following events
all occur:\\
$\bullet$ $Y_1 \in \cT^B_1$ and $Y_2 \in \cT^B_2$;\\
$\bullet$ $Y_1$ is a descendant of $W_1^{(r_1)}$ for some $(9/12) \delta n \le r_1 \le h_u-n^*$;\\
$\bullet$ $Y_2$ is a descendant of $W_2^{(r_2)}$ for some $(9/12) \delta n \le r_2 \le h_u-n^*$;\\
$\bullet$ $(Y_1,Y_2)$ intersect-well;\\
$\bullet$ $h_u \le h(Y_1) = h(Y_2)$.\\
Observe that the last two conditions require in particular that 
$(5/6) \delta n \le h_u \le h(Y_1) = h(Y_2) \le \delta n$, and $\Phi_1(Y_1) = \Phi_2(Y_2)$.
We define the random set:
\eqnst
{ \hI_{V,U}
  = \begin{cases}
    \left\{ (Y_1, Y_2) : \text{$(Y_1,Y_2)$ is a good extra intersection} \right\} 
       & \text{when $\cIW'$ occurs;} \\
    \es & \text{when $\cIW'$ does not occur,}
    \end{cases} }
and the random variable:
\eqnst
{ J
  = J(c'_0)
  = \sum_{\substack{V_1, V_2 \in T_M \\ U_1, U_2 \in T_M}}
     \bone_{ \left| \hI_{V,U} \right| \ge c'_0 \frac{\sigma^4}{D^d} \log (\delta n) }. }
A crucial property is that we have the following inclusion of events:
\eqnsplst
{ \{ J > 0 \} 
  &\subset \left\{ \exists V_1, V_2, U_1, U_2 : \left| \hI_{V,U} \right| 
      \ge c'_0 \frac{\sigma^4}{D^d} \log (\delta n) \right\} \\
  &\subset \left\{ | \cI | \ge c'_0 \frac{\sigma^4}{D^d} \log (\delta n) \right\}
  = \cB'(c'_0). }
Hence in order to conclude, we need to show that 
$\bP ( J > 0 ) \ge \ch{c'_1}/\log (\delta n)$ for suitable $c'_0, \ch{c'_1} > 0$
This follows immediately from the theorem and proposition stated below, that are  
the main results of this section.

\begin{theorem}
\label{thm:hI-moments}
Assume $d=6$ and $\|x\| \le \sqrt{\delta n}$. There exists $n'_9 = n'_9(\sigma^2, C_3, \bp^1) < \infty$
such that for $\delta n \ge n'_9$ and all $V_1, V_2, U_1, U_2 \in T_M$ we have
\eqnst
{ \E \Big[ \left| \hI_{V,U} \right| \,\Big|\, \cIW'(V_1,V_2,U_1,U_2) \Big]
  \ge \frac{c' \sigma^4}{D^6} \log (\delta n), }
and 
\eqnst
{ \E \Big[ \left| \hI_{V,U} \right|^2 \,\Big|\, \cIW'(V_1,V_2,U_1,U_2) \Big]
  \le \frac{C' \sigma^8}{D^{12}} \log^2 (\delta n). }
In particular, taking $c'_0 = \frac{1}{2} c'$, we have
\eqn{e:hI-conclusion}
{ \bP \Big( \left| \hI_{V,U} \right| \ge c'_0 \frac{\sigma^4}{D^6} \log (\delta n) \,\Big|\, 
      \cIW'(V_1,V_2,U_1,U_2) \Big)
  \ge 4 \frac{(c')^2}{C'}
  > 0. }
\end{theorem}

\begin{proposition}
\label{prop:J-moments}
Assume $d=6$ and $\|x\| \le \sqrt{\delta n}$. For $\delta n \ge n'_9$ 
and with the choice of $c'_0$ in the conclusion \eqref{e:hI-conclusion}
of Theorem \ref{thm:hI-moments}, we have
\eqnst
{ \E J(c'_0)
  \ge \frac{c' \sigma^4}{D^6}, }
and
\eqnst
{ \E J(c'_0)^2
  \le \frac{C' \sigma^8}{D^{12}} \log (\delta n). }
In particular, we have
\eqnst
{ \bP ( \cB'(c'_0) )
  \ge \bP ( J(c'_0) > 0 )
  \ge \frac{(c')^2}{C'} \frac{1}{\log (\delta n)}. }
\end{proposition}

\begin{proof}[Proof of Proposition \ref{prop:J-moments} assuming Theorem \ref{thm:hI-moments}]
Due to \eqref{e:hI-conclusion} of Theorem \ref{thm:hI-moments}, we have
\eqnsplst
{ \E J(c'_0)
  &= \sum_{\substack{V_1, V_2 \\ U_1, U_2}} \bP ( \cIW'(V_1,V_2,U_1,U_2) ) \,
    \bP \Big( \left| \hI_{V,U} \right| \ge c'_0 \frac{\sigma^4}{D^6} \log (\delta n) \,\Big|\, 
    \cIW'(V_1,V_2,U_1,U_2) \Big) \\
  &\ge c' \, \sum_{\substack{V_1, V_2 \\ U_1, U_2}} \bP ( \cIW'(V_1,V_2,U_1,U_2) )
  = c' \, \E |\cI'|. }
Due to Theorem \ref{thm:I-moments}, this is at least $c' \sigma^4 D^{-6}$, proving the
first statement.

For the second statement we use that $J(c'_0) \le |\cI|$, and hence the upper bound
on $\E J(c'_0)^2$ follows immediately from Theorem \ref{thm:I-moments}.

The last statement follows from the Paley-Zygmund inequality:
\eqnst
{ \bP ( J(c'_0) > 0 )
  \ge \frac{[\E J(c'_0)]^2}{\E J(c'_0)^2}. }
\end{proof}

Hence it remains to prove Theorem \ref{thm:hI-moments}, that we do in the 
next two sections.

\subsection{Lower bound on the first moment}

We start with the arguments for the lower bound on the first moment.
For the next lemma, assume that the event $\cIW'(V_1,V_2,U_1,U_2)$ occurs.
Recall that this defines vertices $Z_1, Z_1^+ \in \cT_1$ and $Z_2, Z_2^+ \in \cT_2$ at levels 
$k_1, k_1+1$ and $k_2, k_2+1$, respectively, as well as the 
vertices $W_1^{(r_1)}$ at levels $k_1+1 \le r_1 \le h_u$ and 
$W_2^{(r_2)}$ at levels $k_2+1 \le r_2 \le h_u$. We will write 
$\Phi(W^{(r_1)}_1) = (w^{(r_1)}_1, r_1)$ and $\Phi(W^{(r_r)}_r) = (w^{(r_2)}_2, r_2)$.

\begin{lemma}
\label{lem:ts-lemma}
There exists an absolute constant $c' > 0$ with the following property.
Assume that the event $\cIW'(V_1,V_2,U_1,U_2)$ occurs. 
Then for $\delta n \ge 12 n_2(\p_1, \eps = 1/2, L=1)$ and any 
$(9/12) \delta n \le r_1, r_2 \le h_u-n_2$ we have
\eqn{e:ts-bnds}
{ \bP \Big( \| w^{(r_j)}_j - z_j^+ \| \le \frac{1}{\sqrt{2}} \sqrt{r_j - k_j - 1},\, 
     \| u_j - w^{(r_j)}_j \| \le \frac{1}{8} \sqrt{h_u - r_j},\, j = 1, 2 \,\Big|\, 
     \cIW' \Big)
  \ge c'. }
\end{lemma}

\begin{proof}
Condition on the event $\cIW'$, and let us further condition on the spatial
locations $z_1^+$, $z_2^+$ and $u_1 = u_2 = u$.
Due to Lemma \ref{lem:distrib}, the conditional distribution of the path 
between $(z_j^+,k_j+1)$ and $(u, h_u)$, is a 
random walk started at $z_j^+$ and conditioned to arrive at $u$
at time $h_u - k_j - 1$, $j = 1,2$. Let us write 
\eqnst
{ \Omega_j 
  = \left\{ w \in \Z^d : \| w - z_j^+ \| \le \frac{1}{\sqrt{2}} \sqrt{r_j - k_j - 1},\, 
     \| u - w \| \le \frac{1}{8} \sqrt{h_u - r_j} \right\}. }
Therefore, the probability in the statement of the lemma equals
\eqnsplst
{ \prod_{j \in \{ 1, 2 \}} \, \frac{\sum_{w \in \Omega_j} \bp^{r_j - k_j - 1}(z_j^+,w) \, 
     \bp^{h_u - r_j}(w,u)}{\bp^{h_u - k_j - 1}(z_j^+,u)}. }
Since $h_u - k_j - 1 \ge (5/6) \delta n - (4/6) \delta n - 1 \ge (1/12) \delta n \ge n_2$, 
and $h_u - r_j \ge n_2$, the local limit theorem (Lemma \ref{lem:LCLT}) 
implies that there exists $C = C(d)$ and $c = c(d) > 0$, such that for all $w \in \Omega_j$
we have  
\eqnspl{e:local-clt-bnds}
{ \bp^{h_u - k_j - 1}(z_j^+,u)
  &\le \frac{C}{D^d \, (h_u - k_j - 1)^{d/2}}
  \le \frac{C}{D^d \, (\delta n)^{d/2}} \\
  \bp^{r_j - k_j - 1}(z_j^+,w)
  &\ge \frac{c}{D^d \, (r_j - k_j - 1)^{d/2}}
  \ge \frac{c}{D^d \, (\delta n)^{d/2}} \\
  \bp^{h_u - r_j}(w,u)
  &\ge \frac{c}{D^d \, (h_u - r_j)^{d/2}}. }
Since we are conditioning on $\cIW'$, the restriction 
$\| u - z_j^+ \| \le \frac{1}{2} \sqrt{h_u - k_j - 1}$ holds. We claim that this
implies that 
\eqn{e:Omegaj-bnd}
{ | \Omega_j | 
  \ge c' D^d (h_u - r_j)^{d/2}. }
Indeed, writing $t = h_u - k_j - 1$, $\alpha t = h_u - r_j$, $(1 - \alpha)t = r_j - k_j - 1$,
we have $\alpha \le 1/2$, and this implies the inequalities 
$\frac{1}{8} \sqrt{\alpha t} \le \frac{1}{\sqrt{2}} \sqrt{(1-\alpha)t}$ and
$\frac{1}{2} \sqrt{t} \le \frac{1}{\sqrt{2}} \sqrt{(1-\alpha)t}$. These in turn imply that 
at least half of the ball of radius $\alpha t$ centred at $u$ (namely the half lying in the
direction of $z_j^+$) is included in $\Omega_j$.
Putting together the bounds \eqref{e:local-clt-bnds} and \eqref{e:Omegaj-bnd} gives the
statement of the lemma.
\end{proof}

For the next lemma, assume the event $\cIW'(V_1,V_2,U_1,U_2)$, and consider
vertices $Y_1 \in \cT^B_1$ and $Y_2 \in \cT^B_2$ at common height 
$h_u \le h(Y_1) = h_y = h(Y_2) \le \delta n$. Let $Y_1$ 
be a descendant of $W_1^{(r_1)}$, and $Y_2$ be a descendant of
$W_2^{(r_2)}$. 
%We shall say that $Y_1, Y_2$ are \emph{well-placed}, if the following holds:\\
%\emph{Case 1.} If $(5/6) \delta n \le h_u \le (11/12) \delta n$ we require that 
%$h_y \ge h_u$.\\
%\emph{Case 2.} If $(11/12) \delta n < h_u \le \delta n$, we require that 
%$h_u - (1/24) \delta n \le h_y \le h_u$ and $r_1, r_2 \le 2 h_y - h_u - 1$.

\begin{lemma}
\label{lem:Y-good-extra}
Fix $V_1, V_2, U_1, U_2$, and assume the event $\cIW'(V_1,V_2,U_1,U_2)$. Let 
$Y_1 \in \cT^B_1$ and $Y_2 \in \cT^B_2$ be vertices at height 
$h_u \le h(Y_1) = h_y = h(Y_2) \le \delta n$, and assume they 
are descendants of $W_1^{(r_1)}$ and $W_2^{(r_2)}$, respectively. 
Assume also that $h_y - r_1, h_y - r_2 \ge n_2(\p^1, \eps = 1/2, L = 1)$. 
We have
\eqnspl{e:Y-intersect}
{ \bP \big( (Y_1,Y_2) \in \hI_{V,U} \,\big|\, \cT^1,\, \cT^2,\, \cIW'(V_1,V_2,U_1,U_2) \big)
  \ge \frac{c'}{D^d} \frac{1}{(2 h_y - r_1 - r_2)^{d/2}}. }
\end{lemma}

\begin{proof}
Without loss of generality, we assume that $h_y - r_1 \le h_y - r_2$ 
(the opposite case is handled analogously). Condition on the event in the statement,
and let us further condition on the spatial locations of $W_1^{(r_1)}$ and
$W_2^{(r_2)}$, that we denote by $w_1$ and $w_2$, for short. Due to 
Lemma \ref{lem:ts-lemma}, we may assume, at the cost of a constant factor, 
that the event in \eqref{e:ts-bnds} hold. Assuming that this is the case,
let 
\eqnst
{ \Omega
  = \{ y \in \Z^d : \| w_1 - y \| \le \frac{1}{\sqrt{2}} \sqrt{h_y - r_1},\, 
    \| w_2 - y \| \le \frac{1}{\sqrt{2}} \sqrt{h_y - r_2} \}. }
We show that $|\Omega| \ge c D^d (h_y - r_1)^{d/2}$.
For this, it is enough to show that all points $y$ satisfying the condition 
on $\| w_1 - y \| \le \frac{1}{8} \sqrt{h_y - r_1}$ automatically satisfy the condition
on $\| w_2 - y \|$ in the definition of $\Omega$.

%\emph{Case 1.}  
Write $a = h_u - r_1$, $b = h_u - r_2$, $c = h_y - h_u$,
so that $0 \le a \le b$ and $c \ge 0$. Then we have 
\eqnst
{ \| w_2 - w_1 \|
  \le \frac{1}{8} \sqrt{a} + \frac{1}{8} \sqrt{b}
  \le \frac{1}{4} \sqrt{a + b}. }
Hence we are left to show that 
\eqnst
{ \frac{1}{8} \sqrt{a + c} + \frac{1}{4} \sqrt{a + b} 
  \le \frac{1}{\sqrt{2}} \sqrt{b + c}. }
It is easy to see that this follows from $a+c \le b+c$ and 
$a+b \le 2(b+c)$.

%\emph{Case 2.} Write $a = h_u - r_1$, $b = h_u - r_2$, $c = h_u - h_y$,
%so that $c \ge 0$ and $a-c \ge c$ and $b-c \ge c$. We again have
%$\| w_2 - w_1 \| \le \frac{1}{4} \sqrt{a + b}$, and we are left to show that
%\eqnst
%{ \frac{1}{8} \sqrt{a - c} + \frac{1}{4} \sqrt{a+b}
%  \le \frac{1}{\sqrt{2}} \sqrt{b-c}. }
%This easily follows from $a - c \le b - c$ and $a+b \le 2(a-c) + 2(b-c) \le 4 (b-c)$.

With the estimate on the size of $\Omega$ at hand, and using that 
$h_y - r_1, h_y - r_2 \ge n_2$, we can apply the 
local CLT (Lemma \ref{lem:LCLT}) to get
that the conditional probability in \eqref{e:Y-intersect} is at least:
\eqnsplst
{ \sum_{y \in \Omega} \p^{h_y - r_1}(w_1,y) \, \p^{h_y - r_2}(w_2,y)
  &\ge |\Omega| \, \frac{c}{D^d} (h_y - r_1)^{-d/2} \, \frac{c}{D^d} (h_y - r_2)^{-d/2} \\
  &\ge \frac{c}{D^d} (h_y - r_2)^{-d/2}
  \ge \frac{c}{D^d} (h_y - r_2 + h_y - r_1)^{-d/2}, }
as claimed.
\end{proof}

\begin{lemma}
Assume $d = 6$. Then for $\delta n \ge  (24 n^*)^2$
(where $n^* = \max \{ n_1, n_2(\p^1, \eps = 1/2, L = 1) \}$), we have
\eqnst
{ \E \big( |\hI_{V,U}| \,\big|\, \cIW' \big)
  \ge \frac{c' \, \sigma^4}{D^d} \log (\delta n). }
\end{lemma}

\begin{proof}
%We treat the following two cases slightly differently:\\
%\emph{Case 1: $(5/6) \delta n \le h_u \le (11/12) \delta n$;} \\
%\emph{Case 2: $(11/12) \delta n < h_u \le \delta n$.}\\
%
%\emph{Case 1.} 
Due to Lemma \ref{lem:Y-good-extra}, we have
\eqnspl{e:1st-moment-lb}
{ \E \big( |\hI_{V,U}| \,\big|\, \cIW' \big)
  \ge \frac{c'}{D^d} \sum_{r_1 = (9/12) \delta n}^{h_u-n_2}
      \sum_{r_2 = (9/12) \delta n}^{h_u-n^*} 
      \sum_{h_y = h_u}^{\delta n} \, \frac{1}{(2 h_y - r_1 - r_2)^3}
      \E \cL(h_y,r_1) \E \cL(h_y,r_2), }
where, analogously to \cite[Lemma 3.10]{JN14}, $\cL(h_y,r_1)$ denotes the number
of vertices $Y_1$ of $\cT^B_1$ at level $h_y$ that are descendants of $W^{(r_1)}_1$.
As in \cite[Lemma 3.10]{JN14}, we have
$\E \cL(h_y,r_1) \ge c \sigma^2$ and $\E \cL(h_y,r_2) \ge c \sigma^2$.
This gives that the right hand side of \eqref{e:1st-moment-lb} is at least
\eqnsplst
{ \frac{c' \, \sigma^4}{D^d} \sum_{r_1 = (9/12) \delta n}^{h_u-n^*}
      \sum_{r_2 = (9/12) \delta n}^{h_u-n^*} 
      \sum_{h_y = h_u}^{\delta n} \, \frac{1}{(2 h_y - r_1 - r_2)^3}. }
Let us write $s_1 = h_u - r_1$, $s_2 = h_u - r_2$ and $h^* = h_y - h_u$, 
so that the last expression satisfies
\eqnsplst
{ &\ge \frac{c' \, \sigma^4}{D^d} \sum_{s_1 = n^*}^{(1/24) \delta n}
      \sum_{s_2 = n^*}^{(1/24) \delta n} 
      \sum_{h^* = 0}^{(1/12) \delta n} \, \frac{1}{(2 h^* + s_1 + s_2)^3} \\
  &\ge \frac{c' \, \sigma^4}{D^d} \sum_{s_1 = n^*}^{(1/24) \delta n}
      \sum_{s_2 = n_2}^{(1/24) \delta n} \, \frac{1}{(s_1 + s_2)^2} \\
  &\ge \frac{c' \, \sigma^4}{D^d} \sum_{s_1 = n^*}^{(1/24) \delta n}
      \, \frac{1}{s_1} \\ 
  &\ge \frac{c' \, \sigma^4}{D^d} \log (\delta n), }
using in the last step that $\log (n^*) + \log (24) \le \frac{1}{2} \log (\delta n)$.
\end{proof}

\subsection{Upper bound on the second moment}

We fix $V_1, V_2, U_1, U_2$ such that $(5/6) \delta n \le h(U_1) = h_u = h(U_2) \le (11/12) \delta n$,
and recall that we condition on the event $\cIW'$.
Fix heights $h_u \le h_y, h_{\ty} \le \delta n$, and consider a pair of vertices 
$(Y_1, \tY_1)$ that are both in $\cT^B_1$ such that $h(Y_1) = h_y$ and $h(\tY_1) = h_{\ty}$.
There then exist unique heights $(9/12) \delta n \le r_1, \tr_1 \le h_u - n_1$ such that 
$Y_1 \succ W_1^{(r_1)}$ and $\tY_1 \succ W_1^{(\tr_1)}$. 
Let $\tZ_1$ denote the highest common ancestor of $Y_1$ and $\tY_1$, and 
let $k_1' = h(\tZ_1)$. Note that we can have $r_1 = \tr_1$, in which case
$k_1' \ge r_1 = \tr_1$, while if $r_1 \not= \tr_1$, we have $k_1' = r_1 \wedge \tr_1$.
Let us write $\cL'(h_y, h_{\ty}, k_1', r_1, \tr_1)$ for the number of pairs
$(Y_1, \tY_1)$ that satisfy the above height restrictions with given 
$h_y, h_{\ty}, k_1', r_1, \tr_1$. The following is an analogue of \cite[Lemma 3.11]{JN14}.

\begin{lemma}
\label{lem:L-exp}
We have
\eqnst
{ \E \Big[ \cL'(h_y, h_{\ty}, k'_1, r_1, \tr_1) \,\Big|\, \cIW' \Big]
  \le \begin{cases}
      \sigma^4 & \text{when $r_1 = \tr_1 < k_1' < h_y, h_{\ty}$;} \\
      \sigma^2 & \text{when $r_1 = \tr_1 < k_1' = h_y \wedge h_{\ty}$;} \\
      C_3      & \text{when $r_1 = \tr_1 = k_1' < h_y, h_{\ty}$;} \\
      \sigma^4 & \text{when $r_1 \not= \tr_1$, $k_1' = r_1 \wedge \tr_1 < h_y, h_{\ty}$;} \\
      \end{cases} }
%  \le (C_3 + 2 \sigma^4 \delta n) \bone_{h_y > k'_1,\, h_{\ty} > k'_1}
%      + (1 + 2 \sigma^2 \delta n) \bone_{h_y = k'_1 \text{ or } h_{\ty} = k'_1}. }
\end{lemma}

(Observe that there is no case $r_1 = \tr_1 = k_1' = h_y \wedge h_{\ty}$, since
$r_1 = \tr_1 < h_u \le h_y, h_{\ty}$.)

\begin{proof}
Conditional on $\cIW'$ the distribution of $\cT^B_1$ is the same as that 
of $\cT_{h_u - k_1 - 1, 2 \delta n - k_1 -1}$ (cf.~Lemma \ref{lem:distrib}). Hence
the proof boils down to the same (straightforward) branching process calculations as 
the proof of \cite[Lemma 3.11]{JN14}. 
\end{proof}

The following `diagrammatic estimate' is taken without change from \cite{JN14}.
Recall the constant $n_1 = n_1(\p_1)$ from Lemma \ref{lem:LCLT}(i), and the 
constant $L_1 = L_1(\p_1)$ from \eqref{e:Green-ub}.
See Figure \ref{fig:f-diagramm}(a).

\begin{lemma}[{\cite[Lemma 3.12]{JN14}}]
\label{lem:f-bnd}
Suppose $d \ge 3$. There are constants $C = C(d) > 0$
and $C_2 = C_2(\p^1)$ such that for all $\tz_1, \tz_2 \in \Z^d$ we have
\eqnst
{ \sum_{h : k'_1 \vee k'_2 \le h \le \delta n} \p^{2h - k'_1 - k'_2}(\tz_1,\tz_2)
  \le \frac{C}{D^d} f(k'_1,k'_2,\tz_1,\tz_2) \, ,}
where
\eqnst
{ f(k'_1,k'_2,\tz_1,\tz_2)
  := \begin{cases}
     |k'_1 - k'_2|^{(2-d)/2}
        & \parbox{4.5cm}{if $\| \tz_1 - \tz_2 \| \le |k'_1 - k'_2|^{1/2}$
        and $|k'_1 - k'_2| \ge n_1$;} \\
     & \\
     C_2
        & \text{if $\| \tz_1 - \tz_2 \| \le |k'_1 - k'_2|^{1/2} < \ch{\sqrt{n_1}}$;} \\
     & \\
    \| \tz_1 - \tz_2 \|^{2-d}
        & \parbox{4.5cm}{if $\| \tz_1 - \tz_2 \| > |k'_1 - k'_2|^{1/2}$
        and $\| \tz_1 - \tz_2 \| \ge L_1$;} \\
     & \\
     C_2
        & \text{if $|k'_1 - k'_2|^{1/2} < \| \tz_1 - \tz_2 \| < L_1$.}
     \end{cases} }
\end{lemma}

\psfrag{0}{$0$}
\psfrag{k1}{$k'_1$}
\psfrag{k2}{$k'_2$}
\psfrag{h}{$h$}
\psfrag{hu}{$h_y$}
\psfrag{hw}{$h_{\ty}$}
\psfrag{z1}{$\tz_1$}
\psfrag{z2}{$\tz_2$}
\psfrag{u}{$y$}
\psfrag{w}{$\ty$}
\psfrag{deltan}{$\delta n$}

\begin{figure}
\begin{center}
\includegraphics[scale=0.6]{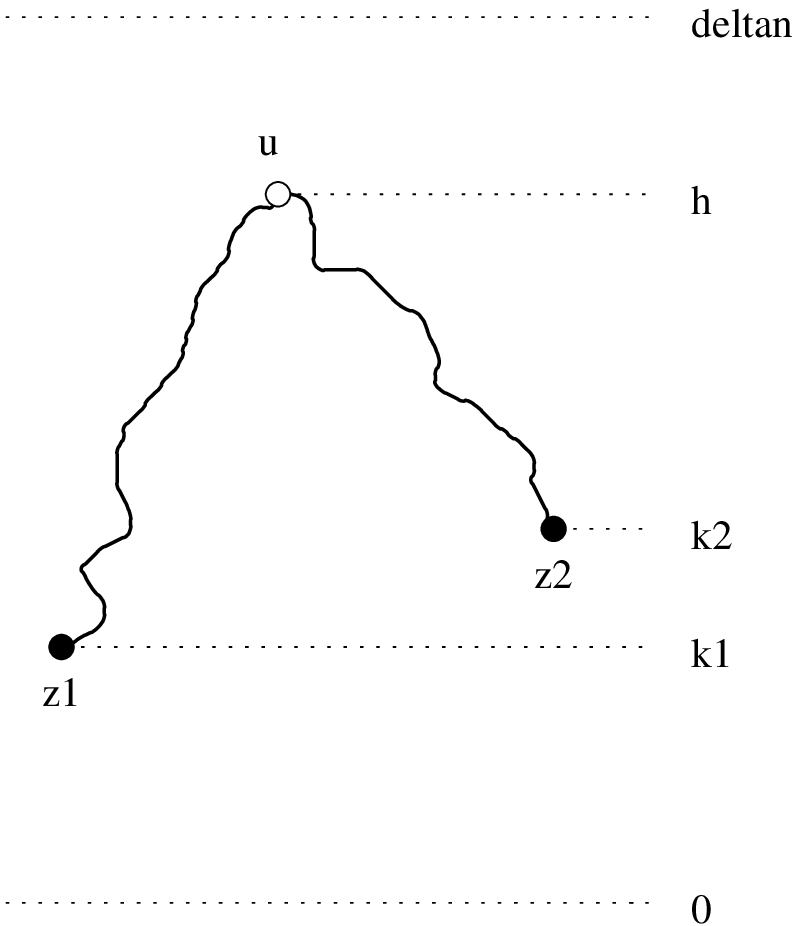}
\hskip1cm
\includegraphics[scale=0.6]{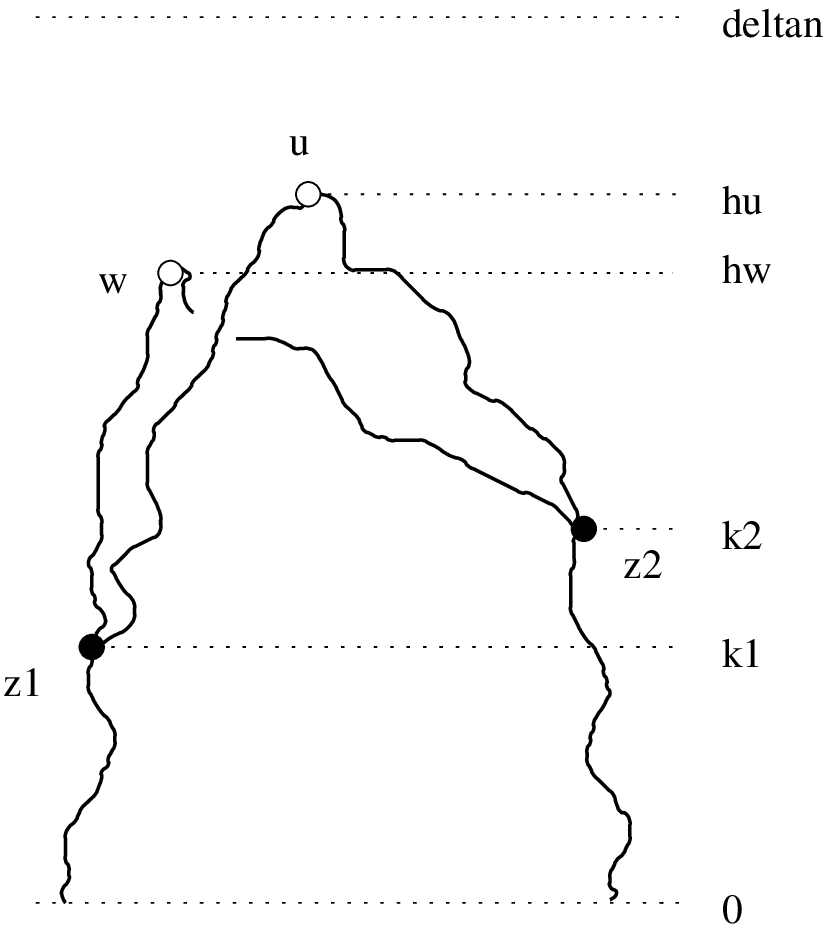}
\end{center}
\caption{(a) Illustration of the quantity bounded in Lemma \ref{lem:f-bnd}.
The curves represent random walk transition probabilities between
the indicated space-time points.
The expression is summed over $y$ to obtain $\p^{2h - k'_1 - k'_2}(\tz_1,\tz_2)$
and then summed over $h$. 
When $\| \tz_1 - \tz_2 \| > |k'_1 - k'_2|^{1/2}$, we get the
Green function decay from the spatial separation $\| \tz_1 - \tz_2 \|$.
When $\| \tz_1 - \tz_2 \| \le |k'_1 - k'_2|^{1/2}$ we get decay
from the time separation $|k'_1 - k'_2|$.
(b) Illustration of the quantity appearing in $q_{a,1}$ that contains ``two copies'' of $f$.
[Figure adapted by permission of Springer, from \emph{Commun.~Math.~Phys.} 
\textbf{331} (2014), 67--109 (Electrical resistance of the low-dimensional 
critical branching random walk. A.A.~J\'arai and A.~Nachmias) \CircledTop{c} (2014).]}
\label{fig:f-diagramm}
\end{figure}

\begin{proposition}
\label{prop:hI^2-bnd}
Assume $d = 6$. We have
\eqnst
{ \E \big( |\hI_{V,U}|^2 \,\big|\, \cIW' \big)
  \le \frac{C' \, \sigma^8}{D^{2d}} \log^2 (\delta n). }
\end{proposition}

\begin{proof}
The proof broadly follows the outline of the proof of \cite[Theorem 3.8]{JN14}.
In addition to the event $\cIW'$, let us further condition on
the spatial location $u$ of the intersection $(U_1, U_2) \in \cI'$,
as well as the spatial locations $z_1^+$ and $z_2^+$.
Let $Y_1, \tY_1 \in \cT^B_1$ and $Y_2, \tY_2 \in \cT^B_2$ be pairs of
tree vertices, such that $h(Y_1) = h_y = h(Y_2)$ and $h(\tY_1) = h_{\ty} = h(\tY_2)$.
Recall the notation introduced at the beginning of this section, and 
extend it to the tree $\cT^B_2$; e.g.~$\tZ_2$ is the highest common ancestor of 
$Y_2, \tY_2$ at height $k_2'$, etc.

We give separate bounds in the following four cases.\\
\emph{Case (a).} $r_1 = \tr_1 < k_1'$ and $r_2 = \tr_2 < k_2'$;\\
\emph{Case (b1).} $r_1 = \tr_1 < k_1'$ and $r_2 \wedge \tr_2 = k_2'$;\\ 
\emph{Case (b2).} $r_1 \wedge \tr_1 = k_1'$ and $r_2 = \tr_2 < k_2'$;\\
\emph{Case (c).} $r_1 \wedge \tr_1 = k_1'$ and $r_2 \wedge \tr_2 = k_2'$.\\
We then have
\eqnspl{e:four-terms}
{ &\E \big( |\hI_{V,U}|^2 \,\big|\, \cIW',\, \Phi(U_1) = (u,h_u) = \Phi(U_2),\, 
    \Phi(Z_1^+) = (z_1^+, k_1+1),\, \Phi(Z_2^+) = (z_2^+, k_2+1) \big) \\
  &\qquad = S_a + S_{b1} + S_{b2} + S_c, }
where the four terms represent contributions from intersecting pairs
satisfying the criteria of the respective cases. Then the proposition
follows from the three lemmas below.
\end{proof}

\begin{lemma}
\label{lem:S_a}
We have 
\eqnst
{ S_a 
  \le \frac{C' \, \sigma^8}{D^{2d}} \, \log^2 (\delta n). }
\end{lemma}

\begin{lemma}
\label{lem:S_b}
We have 
\eqnst
{ S_{b1} + S_{b2}  
  \le \frac{C' \, \sigma^8}{D^{2d}} \, \log^2 (\delta n). }
\end{lemma}

\begin{lemma}
\label{lem:S_c}
We have 
\eqnst
{ S_c 
  \le \frac{C' \, \sigma^8}{D^{2d}} \, \log^2 (\delta n). }
\end{lemma}

\begin{proof}[Proof of Lemma \ref{lem:S_a}.] 
By symmetry, we can restrict to $r_1 \ge r_2$. We then have
\eqnsplst
{ S_a &\le 2 \, \sum_{(9/12) \delta n \le r_2 \le r_1 \le h_u-n_1} \,\, 
      \sum_{\substack{r_1 < k_1' \le \delta n \\ 
            r_2 < k_2' \le \delta n}} \,\,
      \sum_{h_y, h_{\ty} = k'_1 \vee k'_2 \vee h_u}^{\delta n} \,
      \E \cL'(h_y, h_{\ty}, r_1, r_1, k'_1) \, \E \cL'(h_y, h_{\ty}, r_2, r_2, k'_2) \\
      &\qquad \times p_a(h_y, h_{\ty}, k'_1, k'_2, r_1, r_2), }
where $p_a$ is the probability, given the conditioning in \eqref{e:four-terms}
that $\Phi(Y_1) = \Phi(Y_2)$ and $\Phi(\tY_1) = \Phi(\tY_2)$. This indeed only
depends on the heights in the argument of $p_a$, and can be written as follows.
Write $\tbp(w_1),\, \tbp(w_2)$ for the conditional distributions of the spatial 
locations $w_1, w_2$, respectively, we have:
\eqnsplst
{ \tbp(w_1)
  &= \frac{\bp^{r_1 - k_1 - 1}(z_1^+, w_1) \, \bp^{h_u - r_1}(w_1,u)}{\bp^{h_u-k_1-1}(z_1^+, u)} 
%  &\le C \, (h_u - r_1)^{-d/2} \, (\delta n)^{d/2} \, \bp^{r_1 - k_1 - 1}(z_1^+, w_1), \quad
   \quad w_1 \in \Z^d; \\
  \tbp(w_2)
  &= \frac{\bp^{r_2 - k_2 - 1}(z_2^+, w_2) \, \bp^{h_u - r_2}(w_2,u)}{\bp^{h_u-k_2-1}(z_2^+, u)} 
%  &\le C \, (h_u - r_2)^{-d/2} \, (\delta n)^{d/2} \, \bp^{r_2 - k_2 - 1}(z_2^+, w_2), \quad
   \quad w_2 \in \Z^d. }
Then we have
\eqnsplst
{ &p_a(h_y, h_{\ty}, k'_1, k'_2, r_1, r_2)
  = \sum_{w_1, w_2 \in \Z^d} \, 
    \sum_{\tz_1, \tz_2 \in \Z^d} \, \sum_{y, \ty \in \Z^d} \,
    \tbp(w_1) \, \tbp(w_2) \, \bp^{k_1'-r_1}(w_1, \tz_1) \, \bp^{k_2'-r_2}(w_2, \tz_2) \\
  &\qquad\qquad \times \bp^{h_y-k_1'}(\tz_1, y) \, \bp^{h_y-k_2'}(\tz_2, y) \, 
    \bp^{h_{\ty}-k_1'}(\tz_1, \ty) \, \bp^{h_{\ty}-k_2'}(\tz_2, \ty) \\
  &\qquad = \sum_{w_1, w_2 \in \Z^d} \, \sum_{\tz_1, \tz_2 \in \Z^d} \,
    \tbp(w_1) \, \tbp(w_2) \, \bp^{k_1'-r_1}(w_1, \tz_1) \, \bp^{k_2'-r_2}(w_2, \tz_2) \\
  &\qquad\qquad \times \bp^{2h_y - k_1' - k_2'}(\tz_1, \tz_2) \, \bp^{2h_{\ty} - k_1' - k_2'}(\tz_1, \tz_2). }
We perform the summation over $h_y, h_{\ty}$ using Lemmas \ref{lem:L-exp} and \ref{lem:f-bnd}.
Restricting the sum to $h_y, h_{\ty} > k_1' \vee k_2'$ we get an upper bound of the form:
\eqnsplst
{ S_{a,1}
  = \frac{C \, \sigma^8}{D^{2d}} \, \sum_{(9/12) \delta n \le r_2 \le r_1 \le h_u-n_1} \,\, 
      \sum_{\substack{r_1 < k_1' \le \delta n \\ 
            r_2 < k_2' \le \delta n}} \,\, q_{a,1}(k_1', k_2', r_1, r_2), }
where 
\eqnsplst
{ q_{a,1}(k_1', k_2', r_1, r_2)
  &= \sum_{w_1, w_2 \in \Z^d} \, \sum_{\tz_1, \tz_2 \in \Z^d} \,
     \tbp(w_1) \, \tbp(w_2) \, \bp^{k_1'-r_1}(w_1, \tz_1) \, \bp^{k_2'-r_2}(w_2, \tz_2) \,
     f(k_1', k_2', \tz_1, \tz_2)^2. }
Similarly, summing over $h_y > k_1' \vee k_2'$ and $h_{\ty} = k_1' \vee k_2'$, and when
the roles of $y$ and $\ty$ are interchanged, yields the upper bound:
\eqnsplst
{ S_{a,2}
  = \frac{C \, \sigma^6}{D^d} \, \sum_{(9/12) \delta n \le r_2 \le r_1 \le h_u-n_1} \,\, 
      \sum_{\substack{r_1 < k_1' \le \delta n \\ 
            r_2 < k_2' \le \delta n}} \,\, q_{a,2}(k_1', k_2', r_1, r_2), }
where 
\eqnsplst
{ q_{a,2}(k_1', k_2', r_1, r_2)
  &= \sum_{w_1, w_2 \in \Z^d} \, \sum_{\tz_1, \tz_2 \in \Z^d} \,
     \tbp(w_1) \, \tbp(w_2) \, \bp^{k_1'- r_1}(w_1, \tz_1) \, \bp^{k_2'-r_2}(w_2, \tz_2) \\
  &\qquad \times \bp^{|k_1'-k_2'|}(\tz_1, \tz_2) \, f(k_1', k_2', \tz_1, \tz_2). }
And finally, our bound when $h_y = h_{\ty} = k_1' \vee k_2'$ is 
\eqnsplst
{ S_{a,3}
  = C \, \sigma^4 \, \sum_{(9/12) \delta n \le r_2 \le r_1 \le h_u-n_1} \,\, 
      \sum_{\substack{r_1 < k_1' \le \delta n \\ 
            r_2 < k_2' \le \delta n}} \,\, q_{a,3}(k_1', k_2', r_1, r_2), }
where 
\eqnsplst
{ q_{a,3}(k_1', k_2', r_1, r_2)
  &= \sum_{w_1, w_2 \in \Z^d} \, \sum_{\tz_1, \tz_2 \in \Z^d} \,
     \bp^{k_1'-r_1}(w_1, \tz_1) \, \bp^{k_2'-r_2}(w_2, \tz_2) \, 
     \bp^{|k_1'-k_2'|}(\tz_1, \tz_2) \, \bp^{|k_1'-k_2'|}(\tz_1, \tz_2). }
This altogether gives an upper bound of the form:
\eqnst
{ S_a
  \le S_{a,1} + S_{a,2} + S_{a,3}, }
and hence we bound each of the three terms separately.

We start with bounding $S_{a,1}$, and we split the summation over 
$\tz_1, \tz_2$ into:\\
(I)+ $\| \tz_2 - \tz_1 \| \le |k_1' - k_2'|^{1/2}$ and $k_2' \ge h_u$;\\
(I)-- $\| \tz_2 - \tz_1 \| \le |k_1' - k_2'|^{1/2}$ and $r_2 < k_2' < h_u$;\\
(II)+ $\| \tz_2 - \tz_1 \| > |k_1' - k_2'|^{1/2}$ and $k_2' \ge h_u$;\\
(II)-- $\| \tz_2 - \tz_1 \| > |k_1' - k_2'|^{1/2}$ and $r_2 < k_2' < h_u$.\\
We start with (I)+, and initially restrict to $|k_1' - k_2'| \ge n_1$. Then 
using Lemma \ref{lem:f-bnd} and the local limit theorem (Lemma \ref{lem:LCLT}) 
to bound $\bp^{k_2' - r_2}(w_2, \tz_2)$, we have
\eqnsplst
{ q_{a,1,(I)+}(k_1', k_2', r_1, r_2)
  &\le \frac{C}{D^d} \, |k_1' - k_2'|^{2-d} \, (k_2' - r_2)^{-d/2} \, \sum_{w_1, w_2 \in \Z^d} \, 
     \sum_{\tz_1 \in \Z^d} \,\, \sum_{\tz_2 : \| \tz_2 - \tz_1 \| \le |k_1' - k_2'|^{1/2}} \\
  &\qquad \tbp(w_1) \, \tbp(w_2) \, \bp^{k_1'-r_1}(w_1, \tz_1) \\
  &\le C \, |k_1' - k_2'|^{2-d/2} \, (k_2' - r_2)^{-d/2} \, \sum_{w_1 \in \Z^d} \, 
     \sum_{\tz_1 \in \Z^d} \tbp(w_1) \, \bp^{k_1'-r_1}(w_1, \tz_1) \\
  &\le C \, |k_1' - k_2'|^{2-d/2} \, (k_2' - r_2)^{-d/2} \\
  &= C |k_1' - k_2'|^{-1} \, (k_2' - r_2)^{-3}. }
Performing the sum over $k_1'$ yields a factor $\log (\delta n)$, while
performing the sums over $k_2'$, $r_2$ and $r_1$ yield:
\eqnst
{ \sum_{(9/12) \delta n \le r_1 \le h_u - n_1} \, \sum_{r_2 \le r_1} \, \sum_{k_2' \ge h_u}
     (k_2' - r_2)^{-3}
  \le C \, \log (\delta n). }
The contribution to (I)+ from $|k_1' - k_2'| < n_1$ is of lower order, since:
\eqnsplst
{ q_{a,1,(I)+} 
  &\le \frac{C \, C_2^2}{D^d} \, (k_2' - r_2)^{-d/2} \, \sum_{w_1, w_2 \in \Z^d} \, 
     \sum_{\tz_1 \in \Z^d} \,\, \sum_{\tz_2 : \| \tz_2 - \tz_1 \| \le |k_1' - k_2'|^{1/2}} \,\,
     \tbp(w_1) \, \tbp(w_2) \, \bp^{k_1'-r_1}(w_1, \tz_1) \\
  &\le C \, C_2^2 \, |k_1' - k_2'|^{3} \, (k_2' - r_2)^{-3}. }
Summing over $k_1'$ yields a factor $n_1^4$, while the rest of the sums contribute
a factor $\log (\delta n)$.

In case (I)-- we again initially restrict to $|k_1' - k_2'| \ge n_1$. We use Lemma \ref{lem:f-bnd}
and we bound $\tbp(w_2)$ using the local limit theorem to get:
\eqnsplst
{ q_{a,1,(I)-}(k_1', k_2', r_1, r_2)
  &\le \frac{C}{D^d} \, |k_1' - k_2'|^{2-d} \, (h_u - r_2)^{-d/2} \, \sum_{w_1, w_2 \in \Z^d} \, 
     \sum_{\tz_1 \in \Z^d} \,\, \sum_{\tz_2 : \| \tz_2 - \tz_1 \| \le |k_1' - k_2'|^{1/2}} \\
  &\qquad \tbp(w_1) \, \bp^{k_1'-r_1}(w_1, \tz_1) \bp^{k_2'-r_2}(w_2, \tz_2) \\
  &\le C \, |k_1' - k_2'|^{2-d/2} \, (h_u - r_2)^{-d/2} \, \sum_{w_1 \in \Z^d} \, 
     \sum_{\tz_1 \in \Z^d} \tbp(w_1) \, \bp^{k_1'-r_1}(w_1, \tz_1) \\
  &\le C \, |k_1' - k_2'|^{2-d/2} \, (h_u - r_2)^{-d/2} \\
  &= C |k_1' - k_2'|^{-1} \, (h_u - r_2)^{-3}. }
Summing over $k_1'$ yields a factor of $\log (\delta n)$, while the sums over 
$k_2'$, $r_2$, $r_1$ yield:
\eqnst
{ \sum_{(9/12) \delta n \le r_1 \le h_u - n_1} \, \sum_{r_2 \le r_1} \, \sum_{r_2 < k_2' < h_u}
     (h_u - r_2)^{-3}
  \le C \, \log (\delta n). }
The contribution to (I)-- from $|k_1' - k_2'| < n_1$ is again of order $n_1^4 \, \log (\delta n)$.

We turn to the bound for (II)+, where we initially restrict to $\| z_1 - z_2 \| \ge L_1$.
We apply the local limit theorem to $\bp^{k_2' - r_2}(w_2, \tz_2)$, which gives
\eqnsplst
{ q_{a,1,(II)+}(k_1', k_2', r_1, r_2)
  &\le \frac{C}{D^d} \, (k_2' - r_2)^{-d/2} \, \sum_{w_1, w_2 \in \Z^d} \, 
     \sum_{\tz_1 \in \Z^d} \,\, \sum_{\tz_2 : \| \tz_2 - \tz_1 \| > |k_1' - k_2'|^{1/2}} \\
  &\qquad \| \tz_2 - \tz_1 \|^{4-2d} \, \tbp(w_1) \, \tbp(w_2) \, \bp^{k_1'-r_1}(w_1, \tz_1) \\
  &\le C \, |k_1' - k_2'|^{2-d/2} \, (k_2' - r_2)^{-d/2} \, \sum_{w_1 \in \Z^d} \, 
     \sum_{\tz_1 \in \Z^d} \, \tbp(w_1) \, \bp^{k_1'-r_1}(w_1, \tz_1) \\
  &\le C \, |k_1' - k_2'|^{-1} \, (k_2' - r_2)^{-3}. }
This is the same expression as in case (I)+, and hence yields $C \log^2 (\delta n)$
upon summation. The contribution from $\| z_1 - z_2 \| < L_1$ is bounded by
$C \, L_1^7 \, \log (\delta n)$. 

In the case (II)--, we again restrict to $\| z_1 - z_2 \| \ge L_1$, and apply the local
limit theorem to $\tbp(w_2)$. This yields:
\eqnsplst
{ q_{a,1,(II)-}(k_1', k_2', r_1, r_2)
  &\le \frac{C}{D^d} \, (h_u - r_2)^{-d/2} \, \sum_{w_1, w_2 \in \Z^d} \, 
     \sum_{\tz_1 \in \Z^d} \,\, \sum_{\tz_2 : \| \tz_2 - \tz_1 \| > |k_1' - k_2'|^{1/2}} \\
  &\qquad \| \tz_2 - \tz_1 \|^{4-2d} \, \tbp(w_1) \, \bp^{k_1'-r_1}(w_1, \tz_1) \, \bp^{k_2'-r_2}(w_2, \tz_2) \\
  &\le C \, |k_1' - k_2'|^{2-d/2} \, (h_u - r_2)^{-d/2} \, \sum_{w_1 \in \Z^d} \, 
     \sum_{\tz_1 \in \Z^d} \, \tbp(w_1) \, \bp^{k_1'-r_1}(w_1, \tz_1) \\
  &\le C \, |k_1' - k_2'|^{-1} \, (h_u - r_2)^{-3}. }
This is the same expression as in case (I)--, and hence we conclude as in that case.
The contribution from $\| z_1 - z_2 \| < L_1$ is again bounded by $C \, L_1^7 \, \log (\delta n)$. 

The above altogether shows that if $\delta n \ge \max \{ e^{n_1^4}, e^{L_1^7} \}$, then 
$S_{a,1} \le C \, \sigma^8 \, D^{-2d} \, \log^2 (\delta n)$.

Similar arguments for $S_{a,2}$ yield the lower order bound:
\eqnst
{ S_{a,2} 
  \le \frac{C \, \sigma^6}{D^{2d}} \, \log (\delta n), }
and for $S_{a,3}$ the lower order bound: 
\eqnst
{ S_{a,3} 
  \le \frac{C \, \sigma^4}{D^{2d}} \, \log (\delta n). }
\end{proof}

\begin{proof}[Proof of Lemma \ref{lem:S_c}.]
We split the summations according to whether:\\
\emph{Subcase (c,1):} $r_1 \not= \tr_1$ and $r_2 \not= \tr_2$;\\
\emph{Subcase (c,2):} either $r_1 = \tr_1$ and $r_2 \not= \tr_2$, or 
  $r_1 \not= \tr_1$ and $r_2 = \tr_2$;\\
\emph{Subcase (c,3):} $r_1 = \tr_1$ and $r_2 = \tr_2$.\\
This gives $S_c = S_{c,1} + S_{c,2} + S_{c,3}$, where we bound each term separately.
Using the last case in Lemma \ref{lem:L-exp} and using Lemma \ref{lem:f-bnd}
to sum over $h_y, h_{\ty}$ we have
\eqnsplst
{ S_{c,1} &\le \sigma^8 \, \sum_{(9/12) \delta n \le r_1 \not= \tr_1 \le h_u-n_1} \,\, 
      \sum_{(9/12) \delta n \le r_2 \not= \tr_2 \le h_u-n_1} \,\, \sum_{h_y, h_{\ty} = h_u}^{\delta n} \,
      q_{c,1}(r_1, \tr_1, r_2, \tr_2, h_y, h_{\ty}), }
where
\eqnsplst
{ q_{c,1}
  &= \sum_{w_1, \tw_1 \in \Z^d} \,
    \sum_{w_2, \tw_2 \in \Z^d} \,
    \tbp(w_1, \tw_1) \, \tbp(w_2, \tw_2) \, \bp^{2h_y - r_1 - r_2}(w_1, w_2) \, 
    \bp^{2h_{\ty} - \tr_1 - \tr_2}(\tw_1, \tw_2), }
with $\tbp(w_1, \tw_1)$ and $\tbp(w_2, \tw_2)$, respectively, denoting the joint distributions 
of the spatial locations $w_1, \tw_1$ and $w_2, \tw_2$, respectively.

We explain the bound in the case when $r_2 \le r_1$ and $\tr_2 \le \tr_1$. There are three 
very similar other cases, that can be handled analogously.
We use the local limit theorem to bound 
\eqnsplst
{ \bp^{2h_y - r_1 - r_2}(w_1, w_2)
  &\le \frac{C}{D^d} (h_y - r_2)^{-d/2} \\
  \bp^{2h_{\ty} - \tr_1 - \tr_2}(\tw_1, \tw_2)
  &\le \frac{C}{D^d} (h_{\ty} - \tr_2)^{-d/2}, }
and then sum over $w_1, w_2, \tw_1, \tw_2$ to get the upper bound:
\eqnsplst
{ &\frac{C \, \sigma^8}{D^{2d}} \, \sum_{r_2 \le r_1 \le h_u - n_1} \, \sum_{h_y \ge h_u} (h_y - r_2)^{-3} \,
     \sum_{\tr_2 \le \tr_1 \le h_u - n_1} \, \sum_{h_{\ty} \ge h_u} (h_{\ty} - \tr_2)^{-3} \\
  &\qquad \le \frac{C \, \sigma^8}{D^{2d}} \, \sum_{r_2 \le r_1 \le h_u - n_1} \, (h_u - r_2)^{-2} \,
     \sum_{\tr_2 \le \tr_1 \le h_u - n_1} \, (h_u - \tr_2)^{-2} \\
  &\qquad \le \frac{C \, \sigma^8}{D^{2d}} \, \sum_{r_2 \le h_u - n_1} \, (h_u - r_2)^{-1} \,
     \sum_{\tr_2 \le h_u - n_1} \, (h_u - \tr_2)^{-1} \\
  &\qquad \le \frac{C \, \sigma^8}{D^{2d}} \, \log^2 (\delta n). }

The subcase $S_{c,2}$ leads to the lower order bound
$S_{c,2} \le C \, C_3 \, \sigma^4 / D^{2d}$, and 
$S_{c,3}$ leads to the lower order bound $S_{c,3} \le C \, C^2_3 / D^{2d}$.
\end{proof}

\begin{proof}[Proof of Lemma \ref{lem:S_b}.]
Due to symmetry, it is enough to bound $S_{b1}$. We show that the bound follows 
from the bounds on $S_a$ and $S_c$ already established. We can write $S_{b1}$
in the form 
\eqnsplst
{ S_{b1} 
  = \sum_{y, \ty \in \Z^d} \, \sum_{h_y, h_{\ty} = h_u}^{\delta n} \,
    g_a(y, \ty, h_y, h_{\ty}) \, g_c(y, \ty, h_y, h_{\ty}), }
where $g_a$ involves summation over the variables $r_1, k_1', w_1, \tz_1$,
and $g_c$ involves summation over the variables $r_2, \tr_2, w_2, \tw_2$.
Applying the Cauchy-Schwarz inequality we get
\eqnsplst
{ S_{b1} 
  &\le \left[ \sum_{y, \ty \in \Z^d} \, \sum_{h_y, h_{\ty} = h_u}^{\delta n} \,
     g_a(y, \ty, h_y, h_{\ty})^2 \right]^{1/2} \,
     \left[ \sum_{y, \ty \in \Z^d} \, \sum_{h_y, h_{\ty} = h_u}^{\delta n} \,
     g_c(y, \ty, h_y, h_{\ty})^2 \right]^{1/2} \\
  &= S_{a}^{1/2} \, S_{c}^{1/2}
  \le \frac{C \sigma^8}{D^{2d}} \, \log^2 (\delta n), }
where we use Lemmas \ref{lem:S_a} and \ref{lem:S_c} in the last step.    
\end{proof}

\section{Analysis of Tree Bad Blocks}
\label{sec:tree-bad}

Most of the analysis of tree bad blocks does not require any change compared 
to \cite[Section 4]{JN14}. However, we need to improve the bounds of
\cite[Lemma 4.1]{JN14} and \cite[Lemma 4.3]{JN14}, as the error terms of
the form $(1 + C_4 \delta)$ present in those lemmas are not sufficient 
for the induction argument in $d = 6$. In fact, it turns out that the 
$C_4 \delta$ term can be completely removed, and we prove this improvement 
below.

%We state the results for the reader's convenience.
For $k \le n$ we define
\eqn{e:bar-gamma}
{ \bar\gamma(k; (x,n))
  = \sum_{y \in \Zd} \frac{\bp^k(o,y) \bp^{n-k}(y,x)}{\bp^n(o,x)} \gamma(k,y). }
For $a = 1, \dots, 6$ we define $\cE_{(a)}$ to be the event that conditions
$(1)$ to $(a-1)$ in Definitions \ref{ktreegood} and \ref{kspatialgood}
are satisfied, but condition $(a)$ is not.

\begin{lemma}[{Strengthening of \cite[Lemma 4.1]{JN14}}]
\label{lem:old-Lemma4.1}
We have
\eqnsplst
{ &\E \big[ \Reff ( \Phi(X_i) \conn \Phi(X_{i+K}) ) \,\Big|\, \cE_{(1)}, \Phi(V_n) = (x,n) \Big] 
  \le K \bar\gamma( \delta n ; (x,n) ). }
%whenever $\delta n \ge \max \{ n_5(C_3), n_7(C_3) \}$.
\end{lemma}

For the proof we will need the following stochastic monotonicity result.

\begin{lemma}
\label{lem:monoton}
Let $i \delta n \le j \le (i+1) \delta n - 1$, and consider the event $\cF(j)$
that $\cT_{n,m}(V_j)$ reaches level $(i+2) \delta n$. Then conditioned on 
$\cF(j)$, the distribution of $\cT_{n,m}(V_j)$ is stochastically larger
than unconditionally, which is stochastically lager than conditionally on 
$\cF(j)^c$.
\end{lemma}

\begin{proof}
We show that the distribution of $\cT_{n,m}(V_j)$ enjoys the FKG property.
This is not immediate, since in the definition \ch{of} $\cT_{n,m}$ we are conditioning 
on the decreasing event that level $m$ is not reached. We realize 
$\cT_{n,m}(V_j)$ on the following probability space. Let $\mathbb{T}_{j,m}$ 
denote an infinitary tree with root $V_j$ and with finite height $m-1-j$
(that is, the vertices at distance $m-1-j$ from $V_j$ are the leaves of 
$\mathbb{T}_{j,m}$). If $V,W \in \mathbb{T}_{j,m}$, and $W$ is a child of $V$,
we write $n(W;V) \ge 1$ for the index of $W$ among all children of $V$.
Let 
\eqnst
{ \Omega_{j,m} 
  = \left\{ (\eta_U) \in \mathbb{N}^{\mathbb{T}_{j,m}} : 
    \parbox{9.5cm}{$\eta_W = 0$ when $W$ is a child of $V$ and $n(W;V) > \eta_V$; and
    $\eta_W = 0$ for all $W$ at height $m-1-j$} \right\}. }
Elements of $\Omega_{j,m}$ are in bijection with finite trees of height at most 
$m-1-j$, where $V_j$ has $\eta_{V_j}$ children, each child $W$ of $V_j$ has,
respectively, $\eta_W$ children, etc. Let us write $T = T(\eta)$ for the tree 
represented by $\eta \in \Omega_{j,m}$. Then we have
\eqnst
{ \bP [ \cT_{n,m}(V_j) = T(\eta) ]
  = \frac{1}{Z_{m,j}} \, \tilde{p}(\eta_{V_j}) \times 
    \prod_{U \in T : U \not= V_j} p(\eta_U), }
where $Z_{m,j}$ is a normalizing constant.

We use the criterion of \cite[Theorem 4.11]{GHM01}. That theorem requires 
that $\eta_U$ be bounded. Therefore, in applying the theorem, 
we first take a large $M$, replace $p(M)$ by $\sum_{k \ge M} p(k)$, and 
apply the theorem to this progeny distribution bounded by $M$. 
Then we let $M \to \infty$. The theorem also requires that the probability
of the maximal element be positive. This is the case, if $p(M) > 0$, which 
we may assume without loss of generality. Another requirement is that 
configurations can be transformed into one another by changing one coordinate at 
a time. This is easily verified: any tree can be transformed into the tree
containing only the root, by removing leaves. It is left to check the 
monotonicity property of one site conditional distributions: 
given any configurations $\zeta$ and $\rho$ on the
vertices $\{ W : W \not= V \}$, such that $\zeta \ge \rho$ we need to check that 
\eqnspl{e:mon}
{ &\bP [ \eta_V \le k \,|\, \eta_W = \zeta_W,\, W \not= V ] \\
  &\qquad \le \bP [ \eta_V \le k \,|\, \eta_W = \rho_W,\, W \not= V ], 
      \quad k = 0, \dots, M. }
We do this by checking some cases separately. 
First note that we may assume without loss of generality, that 
either $V$ is the root, or, if $V'$ is the parent of $V$, that 
$n(V;V') \le \rho_{V'} \le \zeta_{V'}$. Indeed, if $\rho_{V'} < n(V;V')$, then 
the right hand side of \eqref{e:mon} equals $1$ already for $k = 0$
(and hence for all $k$). We may also assume without loss of generality 
that the height of $V$ is less than $m-1-j$, otherwise both sides in 
\eqref{e:mon} are $1$ for $k = 0$ (and hence for all $k$).
Under these assumptions, let 
\eqnsplst
{ k_\zeta
  &= \max \{ j \ge 0 : \text{$\zeta_W > 0$ for the child $W$ of $V$ of index $n(W;V) = j$} \} \\ 
  k_\rho
  &= \max \{ j \ge 0 : \text{$\rho_W > 0$ for the child $W$ of $V$ of index $n(W;V) = j$} \}. }
Since $\zeta \ge \rho$, we have $k_\zeta \ge k_\rho$.
We distinguish the following two cases.
\begin{itemize}
\item[(i)] $k < k_\zeta$: In this case, the left hand side of \eqref{e:mon} is $0$,
and hence \eqref{e:mon} holds.
\item[(ii)] $k \ge k_\zeta \ge k_\rho$: In this case, the left hand side in \eqref{e:mon} is 
\eqnst
{ \frac{1}{Z_\zeta} \sum_{\ell = k_\zeta}^k p(\ell) p(0)^{\ell - k_\zeta}, \qquad \text{ with } \qquad
     Z_\zeta
     = \sum_{\ell = k_\zeta}^M p(\ell) p(0)^{\ell - k_\zeta}, }
and the right hand side is
\eqnst
{ \frac{1}{Z_\rho} \sum_{\ell = k_\rho}^k p(\ell) p(0)^{\ell - k_\rho} \qquad \text{ with } \qquad
    Z_\rho
    = \sum_{\ell = k_\rho}^M p(\ell) p(0)^{\ell - k_\rho}. }
Thus the required inequality \eqref{e:mon} boils down to showing that 
\eqnsplst
{ &\sum_{\ell_1 = k_\rho}^M p(\ell_1) p(0)^{\ell_1-k_\rho} 
       \sum_{\ell_2=k_\zeta}^k p(\ell_2) p(0)^{\ell_2-k_\zeta} \\
  &\qquad \le \sum_{\ell_1 = k_\zeta}^M p(\ell_1) p(0)^{\ell_1-k_\zeta} 
       \sum_{\ell_2=k_\rho}^k p(\ell_2) p(0)^{\ell_2-k_\rho}, }
which reduces to
\eqnst
{ \sum_{\ell_1 = k_\rho}^M p(\ell_1) p(0)^{\ell_1} 
       \sum_{\ell_2=k_\zeta}^k p(\ell_2) p(0)^{\ell_2}
  \le \sum_{\ell_1 = k_\zeta}^M p(\ell_1) p(0)^{\ell_1} 
       \sum_{\ell_2=k_\rho}^k p(\ell_2) p(0)^{\ell_2}. }
This is easily verified, by checking that the terms appearing in the 
left hand side form a subset of the terms appearing in the right hand side.
Hence \eqref{e:mon} also holds in this case.
\end{itemize}
Letting $M \to \infty$ we obtain that the distribution of $\cT_{n,m}(V_j)$
has the FKG property, and this implies the statement of the lemma.
\end{proof}

\begin{proof}[Proof of Lemma \ref{lem:old-Lemma4.1}]
Using the triangle inequality for effective resistance, we have
\eqnsplst
{ &\E \Big[ \Reff ( \Phi(X_i) \conn \Phi(X_{i+K}) ) \,\Big|\, \cE_{(1)},\, \Phi(V_n) = (x,n) \Big] \\
  &\qquad \le \sum_{i'=i}^{i+K-1}
     \E \Big[ \Reff ( \Phi(X_{i'}) \conn \Phi(X_{i'+1}) ) \,\Big|\, 
     \cE_{(1)},\, \Phi(V_n) = (x,n) \Big]. }
For $i' = i+1, \dots, i+K-1$, the conditioning on $\cE_{(1)}$ does not affect 
the distribution of the tree between $X_{i'}$ and $X_{i'+1}$, and hence the sum of the 
terms over those $i'$ are bounded by $(K-1) \bar{\gamma}(\delta n ; (x,n))$.

It is left to bound the $i' = i$ term by $\bar{\gamma}(\delta n ; (x,n))$.

When $\cE_{(1)}$ occurs, we have exactly one of the following three cases:
\begin{itemize}
\item[(i)] There are no levels in $[i \delta n, (i+1) \delta n)$ that reach height
$2 \delta n$,
\item[(ii)] There is more than one such level,
\item[(iii)] There is a unique such level $\ell_1$, but 
$\ell_1 \not\in [(i+1/4) \delta n, (i+3/4) \delta n)]$.
\end{itemize}
Let us handle these cases separately. If (i) occurs, each tree $\cT_{n,m}(V_j)$
for $j = i \delta n, \dots, (i+1) \delta n - 1$ is conditioned to not reach 
level $(i+2) \delta n$, and hence the the part of $\cT_{n,m}$ between 
$X_i$ and $X_{i+1}$ is distributed according to the law of $\cT_{\delta n, 2 \delta n}$.
This gives
\eqnsplst
{ &\E \Big[ \Reff ( \Phi(X_i) \conn \Phi(X_{i+1}) ) \bone_{(i)} \,\Big|\, \Phi(V_n) = (x,n) \Big] \\
  &\qquad \le \bar{\gamma}(\delta n; (x,n)) \, \bP ( (i) \,|\, \Phi(V_n) = (x,n) ), }
since in the definition of $\bar{\gamma}(\delta n ; (x,n))$ we take a supremum over
$m \ge 2 \delta n$.

If (ii) occurs, let $j_1, \dots, j_k$ be the levels whose tree reaches height 
$2 \delta n$ ($k \ge 2$), and denote by $\cF(j_1,\dots,j_k)$ the event that (ii) occurs
with exactly these levels. Due to Lemma \ref{lem:monoton}, on the event $\cF(j_1, \dots, j_k)$,
the trees $\cT_{n,m}(V_{j_s})$, $s = 1, \dots, k$ are stochastically larger than 
unconditionally, and hence stochastically larger than on the event (i).
Since $\Reff$ is a decreasing random variable, we have
\eqnsplst
{ &\E \Big[ \Reff ( \Phi(X_i) \conn \Phi(X_{i+1}) ) \,\Big|\, 
     \cF(j_1,\dots,j_k),\, \Phi(V_n) = (x,n) \Big] \\
  &\qquad \le \E \Big[ \Reff ( \Phi(X_i) \conn \Phi(X_{i+1}) ) \,\Big|\, 
     (i),\, \Phi(V_n) = (x,n) \Big] \\
  &\qquad \le \bar{\gamma}(\delta n; (x,n)). }

Finally, if (iii) occurs, an argument similar to the case (ii) (with $k = 1$) shows that 
\eqnsplst
{ &\E \Big[ \Reff ( \Phi(X_i) \conn \Phi(X_{i+1}) ) \,\Big|\, 
     (iii),\, \Phi(V_n) = (x,n) \Big] \\
  &\qquad \le \bar{\gamma}(\delta n; (x,n)). }
This completes the proof of the lemma.
\end{proof}

\begin{lemma}[{\cite[Lemma 4.2]{JN14}}]
\label{lem:old-Lemma4.2}
\eqnsplst
{ &\E \big[ \Reff ( \Phi(X_i) \conn \Phi(X_{i+K}) ) \,\Big|\, 
      \cE_{(2)},\, \ell_1,\, \Phi(V_n) = (x,n) \Big] \\
  &\quad \bar\gamma( \ell_1 - i \delta n ; (x,n) ) + 1 
      + \bar\gamma( (i+1) \delta n - \ell_1 - 1; (x,n) ) \\
  &\quad\quad + (K-1) \bar\gamma( \delta n ; (x,n) ). }
\end{lemma}

\begin{lemma}[{Strengthening of \cite[Lemma 4.3]{JN14}}]
\label{lem:old-Lemma4.3}
We have
\eqnsplst
{ &\E \big[ \Reff ( \Phi(X_i) \conn \Phi(X_{i+K}) ) \,\Big|\, 
       \cE_{(3)},\, \ell_1,\, \Phi(V_n) = (x,n) \Big] \\
  &\quad \le \bar\gamma( \ell_1 - i \delta n ; (x,n) ) + 1 
      + \bar\gamma( (i+1) \delta n - \ell_1 - 1; (x,n) ) \\
  &\quad\quad + (K-1) \bar\gamma( \delta n ; (x,n) ). }
%whenever $\delta n \ge \max \{ n_5(C_3), n_7(C_3) \}$.
\end{lemma}

\begin{proof}
The modifications required compared to \cite[Lemma 4.3]{JN14} are 
exactly the same as in Lemma \ref{lem:old-Lemma4.1}.
\end{proof}

\begin{lemma}[{\cite[Lemma 4.4]{JN14}}]
\label{lem:old-Lemma4.4}
\eqnsplst
{ &\E \big[ \Reff ( \Phi(X_i) \conn \Phi(X_{i+K}) ) \,\Big|\, 
       \cE_{(4)},\, \ell_1,\, \ell_2,\, \Phi(V_n) = (x,n) \Big] \\
  &\quad \le \bar\gamma( \ell_1 - i \delta n ; (x,n) ) + 1 
      + \bar\gamma( (i+1) \delta n - \ell_1 - 1; (x,n) ) \\
  &\quad\quad + (K-2) \bar\gamma( \delta n ; (x,n) )
      + \bar\gamma( \ell_2 - (i+K-1) \delta n ; (x,n) ) + 1 \\
  &\quad\quad + \bar\gamma( (i+K) \delta n - \ell_2 - 1 ; (x,n) ). }
\end{lemma}

\section{Analysis of Spatially Bad Blocks}
\label{sec:spatially-bad}

The analysis of spatially bad blocks does not require any change
compared to \cite{JN14}. For the convenience of the reader, we
repeat the required definitions, and state the results.

We analyze what happens when condition (5) or (6) in
Definition \ref{kspatialgood} fails, that is, some spatial
displacement is ``not typical'', and also what happens when $\cB'(i,c'_0)$ fails.
We write $\Gtree$ for the event
$$ \Gtree = \{ \hbox{(1)--(4)}, \ell_1, \ell_2, \Phi(V_n) = (x,n) \} \, .$$
We define a set of times
$i \delta n = T_0 < T_1 < \dots < T_{K+4} = (i+K) \delta n$,
time differences $t_1, t_2, \ldots, t_{K+4}$ and
spatial locations $z_0, \ldots , z_{K+4} \in \Z^d$ by
\begin{align*}
z_0     &= x_i            &         &                              & T_0     &= i\delta n\\
z_1     &= v_{\ell_1}     & t_1     &= \ell_1 - i \delta n         & T_1     &= \ell_1 \\
z_2     &= v_{\ell_1 + 1} & t_2     &= 1                           & T_2     &= \ell_1 + 1 \\
z_3     &= x_{i+1}        & t_3     &= (i+1) \delta n - \ell_1 - 1 & T_3     &= (i+1) \delta n \\
z_4     &= x_{i+2}        & t_4     &= \delta n                    & T_4     &= (i+2) \delta n \\
z_5     &= x_{i+3}        & t_5     &= \delta n                    & T_5     &= (i+3) \delta n \\
        &\ \ \vdots       &         &\ \ \vdots                    &         &\ \ \vdots \\
z_{K+1} &= x_{i+K-1}      & t_{K+1} &= \delta n                    & T_{K+1} &= (i+K-1) \delta n \\
z_{K+2} &= v_{\ell_2}     & t_{K+2} &= \ell_2 - (i+K-1) \delta n   & T_{K+2} &= \ell_2 \\
z_{K+3} &= v_{\ell_2+1}   & t_{K+3} &= 1                           & T_{K+3} &= \ell_2 + 1 \\
z_{K+4} &= x_{i+K}        & t_{K+4} &= K \delta n - \ell_2 - 1     & T_{K+4} &= (i+K) \delta n
\end{align*}
Observe that conditional on $\Gtree$, the times $T_s$ and time differences
$t_s$ are non-random but the spatial locations $z_s$ are random.
We define for any $s = 1, \dots, K+4$
\eqnst
{ \bq_s(z)
  = \sum_{\substack{\| y_r \| \le \sqrt{t_r} \\ r = 1, \dots, s-1 \\
    y_1 + \dots + y_{s-1} = z}} \prod_{r=1}^{s-1} \bp^{t_r}(0,y_r). }
For any $s = 1, \dots, K+4$ we define the event
\eqnst
{ \cE^s_{(5)}
  = \bigcap_{r=1}^{s-1} \Big\{ \| z_r - z_{r-1} \| \le \sqrt{t_r} \Big\}
    \bigcap \Big\{ \| z_s - z_{s-1} \| > \sqrt{t_s} \Big\}. }
Note that 
\eqn{e:old-(5.2)} 
{ \bP ( \cE^s_{(5)} \,|\, \cG_{\mathrm{tree}} )
  = \sum_{z,y : \| y \| > \sqrt{t_s}} \bq_s(z) 
    \frac{\bp^{t_s}(z,z+y) \bp^{n-T_s+T_0}(z+y,x)}{\bp^n(o,x)}. }

\begin{lemma}[{\cite[Lemma 5.1]{JN14}}]
\label{lem:5.1}
For any $s = 1, \dots, K+4$ and $s' = 1, \dots, K+4$, the quantity
\eqnst
{ \cR_{s',s}
  = \E \Big[ \Reff ( (z_{s'-1}, T_{s'-1}) \conn ( z_{s'}, T_{s'} )
    \bone_{\cE^s_{(5)}} \,\big|\, \cG_{\mathrm{tree}} \Big], }
satisfies:
\eqnsplst
{ \cR_{s',s}
  &\le \bP ( \cE^s_{(5)} \,|\, \cG_{\mathrm{tree}} ), \quad
    \text{when $s' = 2, K+3$}, \\
  \cR_{s',s}
  &\le \bP ( \cE^s_{(5)} \,|\, \cG_{\mathrm{tree}} ) \gamma(t_{s'}), \quad
    \text{when $s' < s$, $s' \not= 2, K+3$}, \\
  \cR_{s',s}
  &\le \sum_{z} \bq_s(z) \, \sum_{y : \| y \| > \sqrt{t_s}}
    \frac{\bp^{t_s}(o,y) \bp^{n-T_s+T_0}(y,x-z)}{\bp^n(o,x)} \gamma(t_s,y), }
when $s' = s$, $s' \not= 2, K+3$, 
\eqnst
{ \cR_{s',s}
  \le \sum_{\substack{z, y_s, y \\ \| y_s \| > \sqrt{t_s}}} \bq_s(z) 
     \frac{\bp^{t_s}(o,y) \bp^{t_{s'}}(o,y) \bp^{n-t_{s'}-T_s}(y,x-z-y_s)}{\bp^n(o,x)} 
     \gamma(t_{s'},y), }
when $s' > s$, $s' \not= 2, K+3$.
\end{lemma}

\begin{lemma}[{\cite[Lemma 5.2]{JN14}}]
\label{lem:5.2}
\eqnsplst
{ &\E \big[ \Reff ( \Phi(X_i) \conn \Phi(X_{i+K}) ) \,\Big|\, 
       \cE_{(6)} \cup \cE_{(7)},\, \cG_{\mathrm{tree}} \Big] \\
  &\quad \leq (K-2) \gamma(\delta n) + \gamma( \ell_1 - i \delta n ) + 1 
      + \gamma( (i+1) \delta n - \ell_1 - 1 ) \\
  &\quad\quad + \gamma( \ell_2 - (i+K-1) \delta n ) + 1 
      + \gamma( (i+K) \delta n - \ell_2 - 1 ). }
\end{lemma}

Let us write $n$ in the form $n = N K \delta n + K' \delta n + \tilde{n}$,
where $0 \le K' < K$ and $\delta n \le \tilde{n} < 2 \delta n$.
Write $i^{\mathrm{last}} = K N$.

\begin{lemma}[{\cite[Lemma 5.3]{JN14}}]
\label{lem:5.3}
\eqnsplst
{ &\E \Big[ \Reff ( \Phi(X_{i^{\mathrm{last}}}) \conn \Phi(V_n) ) \,\Big|\, \Phi(V_n) = (x,n) \Big] \\
  &\qquad \le K' \bar{\gamma}(\delta n ; (x,n) ) + \bar{\gamma}( \tilde{n} ; (x,n) ). }
\end{lemma}

\section{Analysis of good blocks}
\label{sec:good-blocks}

In this section we estimate expectations of resistances given
$$ \Ggood' = \big \{ \Phi(V_n)=(x,n), \cA(i), \cB'(i,c'_0), \ell_1, \ell_2 \big \} \, .$$
The following lemma is essentially \cite[Lemma 6.1]{JN14}, with $\cGgood$ 
replaced by $\cGgood'$, and requires no modification in its proof.

\begin{lemma}
\label{easyestimates}
Given $\Ggood'$, we have
\begin{enumerate}
\item $\E \big [ \Reff( \Phi(X_i) \lra \Phi(V_{\ell_1}) ) \, \mid \, \Ggood' \big ] \leq \gamma(\ell_1 - i \delta n)$.
\item $\E \big [ \Reff( \Phi(V_{\ell_1}) \lra \Phi(X_{i+1}) ) \, \mid \, \Ggood' \big ] \leq \gamma((i+1) \delta n - \ell_1-1)+1$.
\item For all $i+1 \leq j \leq i+K-2$ we have
$$ \E \big [ \Reff( \Phi(X_j) \lra \Phi(X_{j+1}) ) \, \mid \, \Ggood' \big ] \leq \gamma(\delta n) \, .$$
\item $\E \big [ \Reff( \Phi(X_{i+K-1}) \lra \Phi(V_{\ell_2}) ) \, \mid \, \Ggood' \big ] \leq \gamma(\ell_2 - (i+K-1)\delta n)$.
\item $\E \big [ \Reff( \Phi(V_{\ell_2}) \lra \Phi(X_{i+K}) )  \, \mid \, \Ggood' \big ] \leq \gamma((i+K)\delta n - \ell_2 -1) + 1$.
\item $\E \big [ \Reff( \Phi(V_{\ell_2}) \lra \Phi(X'_{i+K})) \, \mid \, \Ggood' \big ] \leq \gamma((i+K)\delta n - \ell_2 -1) + 1$.
\item $\E \big [ \Reff( \Phi(V_{\ell_1}) \lra \Phi(Y_{i+1}) ) \, \mid \, \Ggood' \big ] \leq \gamma((i+1) \delta n - \ell_1-1)+1$.
\item For all $i+1 \leq j \leq i+K-1$ we have
$$ \E \big [ \Reff( \Phi(Y_j) \lra \Phi(Y_{j+1}) ) \, \mid \, \Ggood' \big ] \leq \gamma(\delta n) \, .$$
\end{enumerate}
\end{lemma}

The following lemma replaces \cite[Lemma 6.2]{JN14}.

\begin{lemma}
\label{lem:short-cut}
Assume $d=6$. There exists $C_5' < \infty$ such that we have
\eqnst
{ \E \Big[ \Reff ( \Phi(X'_{i+K}) \conn \Phi(Y_{i+K}) ) \,\Big|\, \cGgood' \Big]
  \le C_5' \, \max_{1 \le \ell \le \delta n} \gamma(\ell), }
whenever $\delta n \ge \max \{ n_1(\bp^1),\, n'_9(\sigma^2, C_3, \bp^1) \}$.
\end{lemma}

The proof is a straightforward adaptation of the proof of \cite[Lemma 6.2]{JN14}.
We first assume that the progeny distribution is bounded by $M$, prove the statement
with this restriction relying on Theorem \ref{thm:intersect}, and then let $M \to \infty$.
We omit the details.

\begin{proof}[Proof of Theorem \ref{thm:good-blocks}]
This can be completed exactly as in \cite[Section 6]{JN14}.
\end{proof}

\section{Proof of Theorem \ref{thm:main}}

Let $K_0$ be the constant in Theorem \ref{thm:good-blocks}. We fix $K = K_0$ for the
remainder of the proof. Let 
\eqnst
{ n_0 
  = \max \{ n'_3(\sigma^2, C_3, \bp^1, K), \, n'_4(\sigma^2, C_3, \bp^1),\, 4 k_1(\bp^1) \}, }
where $n'_3$ and $n'_4$ are the constants from Theorems \ref{thm:intersect} and 
\ref{thm:good-blocks}, and $k_1$ is the constant from Proposition \ref{prop:variance}.
Let $\delta_0 > 0$, $\xi \in (0,1/2)$ and $A > 0$ be constants.
These will be chosen below in the order: $\delta_0, \xi, A$, and among others 
we will require that 
\eqn{e:old-(7.1)}
{ 2 \delta_0 \le (K+4)^{-1}, \qquad\qquad
  2 \delta_0 \le \ch{\delta_1}, } %\min \{ \delta_1, \delta_2 \}, }
where $\delta_1$ is the constant from Proposition \ref{prop:variance}(ii). 
%and $\delta_2$ is the constant from Lemmas \ref{lem:old-Lemma4.1} and \ref{lem:old-Lemma4.3}.
Once $\delta_0$ and $\xi$ will be chosen, we choose $A$ to satisfy:
\eqn{e:A-ineq}
{ A 
  \ge \max \left\{ \frac{n_0}{\delta_0},\, \frac{1}{\delta_0^2} \right\}, \qquad\quad
  \exp \left( -\frac{1}{2} A^{1/\xi} \right) 
  \le \frac{1}{\sqrt{n_0}}, \qquad\quad
  A^{-1} \le \xi  \frac{n_0}{(\log n_0)^{1 + \xi}}. }
We prove the theorem by induction. Due to the first condition on $A$, 
the theorem holds for $n < \max \{ n_0/\delta_0,\, 1/\delta_0^2 \}$, so we may assume 
$n \ge n_0/\delta_0$ and $n \ge 1/\delta_0^2$. 
Our induction hypothesis is that for all $n' < n$ and for all
$x \in \Zd$ we have
\eqn{e:induct-hyp}
{ \gamma(n',x) 
  \le \begin{cases}
      A n' (\log n')^{-\xi} & \text{when $\| x \| \le \sqrt{n'}$;} \\
      A n' (\log n')^{-\xi} \, 
          \left( 1 -  \frac{\log \left( \|x\|^2 / n' \right)}{\log n'} \right)^{-\xi} 
          & \text{when $\sqrt{n'} < \|x\| \le n'/2$;} \\
      A n' & \text{when $\|x\| >  n'/2$,}
      \end{cases} }
and given the hypothesis, we prove it for $n$. Sometimes, instead of \eqref{e:induct-hyp}, 
it will be convenient to use the following consequence of \eqref{e:induct-hyp} that involves
simpler expressions: there is a universal constant $C > 0$ such that 
\eqn{e:induct-hyp-simple}
{ \gamma(n',x) 
  \le \begin{cases}
      A n' (\log n')^{-\xi} & \text{when $\| x \| \le \sqrt{n'}$;} \\
      A n' (\log n')^{-\xi} \, 
          \left( 1 + \frac{C \xi}{\log n'} \frac{\| x \|^2}{n'} \right)
          & \text{when $\|x\| > \sqrt{n'}$.} 
      \end{cases} }
These bounds follow from the elementary inequalities:
\eqn{e:w-ineq}
{ \left( 1 - \frac{\log ( \| y \|^2/k )}{\log k} \right)^{-\xi}
  \le 1 + \frac{C \xi}{\log k} \frac{\| y \|^2}{k}, 
     \qquad \sqrt{k} < \| y \| \le k/2,\, 0 \le \xi \le 1/2, }
and 
\eqn{e:log-k-ineq}
{ (\log k)^{\xi}
  \le 1 + \frac{C \xi}{\log k} \, \frac{\| y \|^2}{k}, 
     \qquad \| y \| > k/2,\, 0 \le \xi \le 1/2. }

Since $\gamma(n,x) \le n$, we claim that it suffices to prove the statement of the theorem 
when $\| x \| \le n \exp ( - \frac{1}{2} A^{1/\xi} )$. Indeed, suppose that 
$\| x \| > n \exp ( - \frac{1}{2} A^{1/\xi} )$. Then if we have
$\| x \| \le \sqrt{n}$, then also $\exp ( - \frac{1}{2} A^{1/\xi} ) < n^{-1/2}$,
and hence $(\log n)^{-\xi} > A^{-1}$, and hence
\eqnst
{ \gamma(n,x)
  \le n 
  = A n A^{-1}
  < A n (\log n)^{-\xi}. }
On the other hand, if we have $\sqrt{n} < \| x \| \le n/2$, then 
$\| x \|^2 > n^2 \exp ( - A^{1/\xi} )$ implies that 
$A^{-1} < (\log n^2/\| x\|^2)^{-\xi}$, and hence
\eqnsplst
{ \gamma(n,x)
  &\le n
  = A n A^{-1}
  < A n \left( \log n + \log \frac{n}{\| x \|^2} \right)^{-\xi} \\
  &= A n (\log n)^{-\xi} \left( 1 - \frac{\log \| x \|^2/n }{\log n} \right)^{-\xi}. }

Hence from now on we assume the upper bound $\| x \| \le n \exp ( - \frac{1}{2} A^{1/\xi} )$.
Note that due to the second inequality in \eqref{e:A-ineq}, this implies
\eqn{e:x-norm-ineq}
{ \| x \| 
  \le n \exp \left( - \frac{1}{2} A^{1/\xi} \right)
  \le \frac{n}{\sqrt{n_0}}. }
Given such $x$, fix
\eqn{e:new-(7.3)}
{ \delta 
  = \min \Big\{ \eta : \eta \ge \min \left\{ \delta_0,\, \frac{n}{\|x\|^2} \right\},\, 
    \text{$\eta n$ is an integer} \Big\}. }
Observe that 
\eqn{e:delta-ineq}
{ \delta 
  \le \delta_0 + \frac{1}{n}
  \le \delta_0 + \frac{\delta_0}{n_0}
  \le 2 \delta_0, }
and similarly, using \eqref{e:x-norm-ineq}, we have
\eqn{e:delta-ineq2}
{ \delta 
  \le \frac{n}{\| x \|^2} + \frac{1}{n}
  \le \frac{n}{\| x \|^2} + \frac{1}{n_0} \frac{n}{\| x \|^2}
  \le \frac{2 n}{\| x \|^2}. }
We also have
\eqn{e:delta-n-ineq}
{ \delta n 
  \ge \min \left\{ \delta_0 n,\, \frac{n^2}{\| x \|^2} \right\}
  \ge \min \left\{ n_0,\, n_0 \right\}
  = n_0. }
Finally, note that 
\eqnst
{ \| x \|
  \le \sqrt{2n/\delta}, }
which can be seen by considering separately the cases $\delta_0 \le n / \|x\|^2$
and $\delta_0 > n / \|x\|^2$.
Therefore, Theorem \ref{thm:intersect} can be applied to $(x,n)$.

Consider the sequences
\eqnst
{ (0, \dots, K), (K, \dots, 2K), \dots, ((N-1)K, \dots, NK), }
where $n = NK \delta n + K'\delta n + \tilde{n}$, with 
$0 \le K' < K$, $\delta n \le \tilde{n} < 2 \delta n$.
Fix any integer $m \ge 2n$, and define 
\eqnst
{ \gamma_m(n,x)
  = \E_{\cT_{n,m}} \big[ \Reff \big( (o,0) \conn \Phi(V_n) \big) \,\big|\,
    \Phi(V_n) = (x,n) \big], }
where the resistance is considered in the graph $\Phi(\cT_{n,m})$, so that 
$\gamma(n,x) = \sup_{m \ge 2n} \gamma_m(n,x)$.
We bound $\gamma_m(n,x)$ by estimating 
\eqn{e:i-expect}
{ \E \big[ \Reff ( \Phi(X_i) \conn \Phi(X_{i+K}) \,\big|\, \Phi(V_n) = (x,n) \big] }
for each $i = 0, K, 2K, \dots, (N-1)K$ and adding the estimates, using the
triangle inequality for resistance \eqref{e:triangle}, and also adding the 
estimates for the final stretch from $NK \delta n$ to $n$.

Fix $0 \le i \le N-1$. We split the expectation in \ref{e:i-expect} according to 
whether $\cA(i) \cap \cB'(i,c'_0)$ occurred or not. By Theorem \ref{thm:good-blocks}
we have that 
\eqnspl{e:i-expect-bnd}
{ &\E \big[ \Reff ( \Phi(X_i) \conn \Phi(X_{i+K}) ) \bone_{\cA(i) \cap \cB'(i,c'_0)} \,\big|\,
       \Phi(V_n) = (x,n) \big] \\
  &\quad \le \frac{3 K \max_{1 \le k \le \delta n} \gamma(k)}{4} 
       \bP ( \cA(i) \cap \cB'(i,c'_0) \,|\, \Phi(V_n) = (x,n) ) \\
  &\quad \le \frac{3 K A (\delta n) \, (\log (\delta n))^{-\xi}}{4}
       \bP ( \cA(i) \cap \cB'(i,c'_0) \,|\, \Phi(V_n) = (x,n) ), }
where in the last step we used the induction hypothesis.
%It will be convenient to write $\log (\delta n)$ in the form:
%\eqnst
%{ \log (\delta n)
%  = \log \delta + \log n
%  = \log n \left( 1 - \frac{\log(1/\delta)}{\log n} \right), }
%so that we have
%\eqnst
%{ (\log (\delta n))^{-\xi}
%  = (\log n)^{-\xi} \, \left( 1 - \frac{\log(1/\delta)}{\log n} \right)^{-\xi}. }
%Let us abreviate 
%\eqnst
%{ \kappa_\xi(\delta,n)
%  = \left( 1 - \frac{\log(1/\delta)}{\log n} \right)^{-\xi}. }
%Then the right hand side of \eqref{e:i-expect-bnd} can be written
%\eqnspl{e:i-expect-bnd2}
%{ &\E \big[ \Reff ( \Phi(X_i) \conn \Phi(X_{i+K}) ) \bone_{\cA(i) \cap \cB'(i,c_0)} \,\big|\,
%       \Phi(V_n) = (x,n) \big] \\
%  &\quad \le A K (\delta n) (\log n)^{-\xi} \frac{3 \kappa_\xi(\delta,n)}{4}
%       \bP ( \cA(i) \cap \cB'(i,c_0) \,|\, \Phi(V_n) = (x,n) ). }

We now estimate the expectation on the event when either $\cA(i)$ or $\cB'(i,c'_0)$ fail.
Recall that we may write:
\eqnst
{ \cA(i)^c \cup \cB'(i,c'_0)^c
  = \bigcup_{a=1}^6 \cE_{(a)} \cup (\cA(i) \cap \cB'(i,c'_0)^c), }
where $\cE_{(a)}$ was defined in Section \ref{sec:tree-bad}. We need a few lemmas
to estimate the contribution from these terms.

\begin{lemma}
\label{lem:new-Lemma7.1}
There exists $C'_6 > 0$ such that, assuming the induction hypothesis, for any
$\delta n/4 \le k \le 2 \delta n$ we have
\eqnst
{ \bar\gamma(k; (x,n))
  \le \left( 1 + \frac{C'_6 \xi}{\log \delta n} \right) A k (\log k)^{-\xi}, }
where $\bar\gamma$ is defined in \eqref{e:bar-gamma}.
\end{lemma}

\begin{proof}
By the induction hypothesis and its consequence \eqref{e:induct-hyp-simple} we have
\eqnsplst
{ \bar\gamma(k; (x,n))
  &\le A k (\log k)^{-\xi} H_1(\xi), }
where
\eqnsplst
{ H_1(\xi)
  &= \sum_{y : \| y \| \le \sqrt{k}} 
      \frac{\bp^k(o,y) \bp^{n-k}(y,x)}{\bp^n(o,x)} \\
  &\quad + \sum_{y : \| y \| > \sqrt{k}}
      \frac{\bp^k(o,y) \bp^{n-k}(y,x)}{\bp^n(o,x)} 
      \left( 1 + \frac{C \xi}{\log k} \frac{\|y\|^2}{k} \right). }
We have $H_1(0) = 1$ and 
\eqnsplst
{ H'_1(\xi)
  &= \frac{C}{\log k} \sum_{y : \| y \| > \sqrt{k}}
     \frac{\bp^k(o,y) \bp^{n-k}(y,x)}{\bp^n(o,x)} \frac{\| y \|^2}{k} \\
  &\le \frac{C}{\log k} \frac{1}{k} \E \big[ \| S(k) \|^2 \,\big|\, S(n) = x \big]. }
Since $k \ge \delta n / 4 \ge n_0 / 4 \ge k_1$ and 
$\| x \| \le \sqrt{2n/\delta} \le \sqrt{2} n / \sqrt{\delta n} \le 4n / \sqrt{k}$, 
we can apply Proposition \ref{prop:variance}(i) to the expectation on the right 
hand side to see that $H'_1(\xi) \le C/\log k$, and the lemma follows.
\end{proof}

\begin{lemma}
\label{lem:new-Lemma7.2}
There exists $C'_6 > 0$ such that, assuming the induction hypothesis, for all
$s' \ge s$ with $s' \not= 2, K+3$ we have
\eqnsplst
{ &\E \Big[ \Reff ( ( z_{s'-1}, T_{s'-1}) \conn (z_{s'}, T_{s'}) ) \bone_{\cE_{(5)}^s} 
     \,\Big|\, \cG_{\mathrm{tree}} \Big] \\
  &\qquad \le \left( 1 + \frac{C'_6 \xi}{\log \delta n} \right) A t_{s'} (\log t_{s'})^{-\xi}
     \bP ( \cE_{(5)}^s \,|\, \cG_{\mathrm{tree}} ). }
\end{lemma}

\begin{proof}
We first consider the case $s' = s$ and $s' \not= 2, K+3$. 
Appealing to Lemma \ref{lem:5.1} and using consequence \eqref{e:induct-hyp-simple} 
of the induction hypothesis,
we get that the expectation in the claim of the lemma is at most 
$A t_s (\log t_s)^{-\xi} H_2(\xi)$, where
\eqnsplst
{ H_2(\xi)
  = \sum_z \bq_s(z) \sum_{y : \| y \| > \sqrt{t_s}} 
    \frac{\bp^{t_s}(o,y) \bp^{n-T_s+T_0}(y,x-z)}{\bp^n(o,x)}
    \left( 1 + \frac{C \xi}{\log t_s} \frac{\| y \|^2}{t_s} \right). }
By \eqref{e:old-(5.2)}, we have 
\eqnst
{ H_2(0)
  = \bP ( \cE^s_{(5)} \,|\, \cG_{\mathrm{tree}} ). }
Similarly to the previous lemma, we have
\eqn{e:H'_2}
{ H'_2(\xi)
  = \frac{C}{\log t_s} \, \sum_{z} \bq_s(z) \sum_{y : \| y \| > \sqrt{t_s}}
    \frac{\bp^{t_s}(o,y) \bp^{n-T_s+T_0}(z+y,x)}{\bp^n(o,x)} \frac{\| y \|^2}{t_s}. }
Let us multiply and divide by $\bp^{n-T_{s-1}+T_0}(o,x-z)$, which allows us to rewrite
\eqref{e:H'_2} as
\eqnspl{e:H'_2-alt}
{ &\frac{C}{\log t_s} \, \frac{1}{t_s} \sum_{z} \bq_s(z) 
     \frac{\bp^{n-T_{s-1}+T_0}(o,x-z)}{\bp^n(o,x)} \\
  &\qquad \times \sum_{y : \| y \| > \sqrt{t_s}} \| y \|^2 
     \frac{\bp^{t_s}(o,y) \bp^{n-T_s+T_0}(y,x-z)}{\bp^{n-T_{s-1}+T_0}(o,x-z)}. }
Let us now fix $z$. We want to apply Proposition \ref{prop:variance}(ii) to the 
sum over $y$ in \eqref{e:H'_2-alt}. For this, observe that 
\eqnsplst
{ \| x - z \| 
  &\le \| x \| + \| z \|
  \le \sqrt{2n/\delta} + (K+4) \sqrt{\delta n}
  + n \left( \frac{\sqrt{2}}{\sqrt{\delta n}} + \frac{\delta (K+4)}{\sqrt{\delta n}} \right) \\
  &\le n \frac{3}{\sqrt{\delta n}}, }
where in the last step we used \eqref{e:old-(7.1)}. This implies that we have
$\| x- z \| \le 3 n / \sqrt{\delta n} \le 3 n / \sqrt{t_s}$. We also have
$t_s \ge \delta n / 4 \ge n_0 / 4 \ge k_1$, where $k_1$ is the constant in 
Proposition \ref{prop:variance}. In addition, $t_s \le \delta n \le \delta_1 n$,
due to \eqref{e:old-(7.1)}. Hence we can apply Proposition \ref{prop:variance}(ii).
This gives:
\eqnsplst
{ &\sum_{y : \| y \| > \sqrt{t_s}} \| y \|^2 
     \frac{\bp^{t_s}(o,y) \bp^{n-T_s+T_0}(y,x-z)}{\bp^{n-T_{s-1}+T_0}(o,x-z)} \\
  &\qquad = \E \Big[ \| S(t_s) \|^2 \, \bone_{\| S(t_s) \| > \sqrt{t_s}} \,\Big|\, S(n) = x-z \Big] \\
  &\qquad \le C t_s \bP \Big( \| S(t_s) \| > \sqrt{t_s} \,\Big|\, S(n) = x-z \Big) \\
  &\qquad = C t_s \sum_{y : \| y \| > \sqrt{t_s}} 
     \frac{\bp^{t_s}(o,y) \bp^{n-T_s+T_0}(y,x-z)}{\bp^{n-T_{s-1}+T_0}(o,x-z)}. }
Substituting this bound back into \eqref{e:H'_2-alt} and cancelling the 
factors $\bp^{n-T_{s-1}+T_0}(o,x-z)$, we get
\eqnsplst
{ H'_2(\xi)
  &\le \frac{C}{\log t_s} \sum_{\substack{z, y : \\ \| y \| > \sqrt{t_s}}}
      \bq_s(z) \frac{\bp^{t_s}(o,y) \bp^{n-T_s+T_0}(y,x-z)}{\bp^n(o,x)} \\
  &= \frac{C}{\log t_s} \bP ( \cE^s_{(5)} \,|\, \cG_{\mathrm{tree}} ). }
This gives the statement of the lemma in the case $s' = s$.

The case $s' > s$ is similar. Appealing to the last statement of Lemma \ref{lem:5.1}
yields that the required quantity is at most 
$A t_{s'} (\log t_{s'})^{-\xi} H_3(\xi)$, where
\eqnsplst
{ H_3(\xi)
  &= \sum_{\substack{z, y_s, y : \\ \| y_s \| > \sqrt{t_s}}}
    \bq_s(z) \frac{\bp^{t_s}(o,y_s) \bp^{t_{s'}}(o,y) \bp^{n-t_{s'}-T_s}(y,x-z-y_s)}{\bp^n(o,x)} \\
  &\qquad \times \left( 1 + \frac{C \xi}{\log t_{s'}} 
    \frac{\| y \|^2}{t_{s'}} \, \bone_{\| y \| > \sqrt{t_{s'}}}  \right). }
Setting $\xi = 0$ and performing the sum over $y$, we see using \eqref{e:old-(5.2)} that 
$H_3(0) = \bP ( \cE^s_{(5)} \,|\, \cG_{\mathrm{tree}} )$. The derivative 
$H'_3(\xi)$ can be analyzed similarly to $H'_2(\xi)$, this time appealing to 
Proposition \ref{prop:variance}(iii). This gives 
$H'_3(\xi) \le (C/\log t_{s'}) \bP( \cE^s_{(5)} \,|\, \cG_{\mathrm{tree}} )$, 
and proves the statement of the lemma when $s' > s$.
\end{proof}

We now assemble our resistance bounds on the event $\cA(i)^c \cup \cB'(i,c'_0)^c$.
Lemmas \ref{lem:old-Lemma4.1} and \ref{lem:new-Lemma7.1} and the induction hypothesis
yield
\eqnsplst
{ &\E \Big[ \Reff ( \Phi(X_i) \conn \Phi(X_{i+K}) ) \,\Big|\, \cE_{(1)},\, \Phi(V_n) = (x,n) \Big] \\
  &\qquad \le A K \left( 1 + \frac{C'_6 \xi}{\log \delta n} \right)
      \delta n (\log \delta n)^{-\xi}. }
Lemmas \ref{lem:old-Lemma4.2} and \ref{lem:new-Lemma7.1} yield
\eqnsplst
{ &\E \Big[ \Reff ( \Phi(X_i) \conn \Phi(X_{i+K}) ) \,\Big|\, \cE_{(2)},\, \ell_1,\,
      \Phi(V_n) = (x,n) \Big] \\
  &\qquad \le A \left( 1 + \frac{C'_6 \xi}{\log \delta n} \right)
      \Big[ (\ell_1 - i \delta n) (\log (\ell_1 - i \delta n))^{-\xi} + \\
  &\qquad\qquad + ((i+1) \delta n - \ell_1) (\log ((i+1) \delta n - \ell_1))^{-\xi} \\
  &\qquad\qquad + (K-1) \delta n (\log \delta n)^{-\xi} \Big] + 1. }
Writing $\nu = \frac{\ell_1}{\delta n} - i \in [1/4,3/4]$, the first two terms inside the
square brackets can be written as
\eqnsplst
{ &\delta n (\log \delta n)^{-\xi} 
    \left[ \nu \left( \frac{\log (\nu \delta n)}{\log \delta n} \right)^{-\xi}
    + (1 - \nu) \left( \frac{\log (1-\nu) \delta n}{\log \delta n} \right)^{-\xi} \right] \\
  &\qquad = \delta n (\log \delta n)^{-\xi} 
    \left[ \nu \left( 1 + \frac{\log \nu}{\log \delta n} \right)^{-\xi}
    + (1 - \nu) \left( 1 + \frac{\log (1-\nu)}{\log \delta n} \right)^{-\xi} \right] \\
  &\qquad \le \delta n (\log \delta n)^{-\xi} 
    \left[ \nu \left( 1 - \frac{2 \xi \log \nu}{\log \delta n} \right)
    + (1 - \nu) \left( 1 - \frac{2 \xi \log (1-\nu)}{\log \delta n} \right) \right] \\
  &\qquad = \delta n (\log \delta n)^{-\xi}
    \left[ 1 + \frac{2 \xi}{\log \delta n} 
    \left( -\nu \log \nu - (1-\nu) \log (1-\nu) \right) \right] \\
  &\qquad \le \delta n (\log \delta n)^{-\xi}
    \left[ 1 + \frac{2 (\log 2) \xi}{\log \delta n} \right]. }
Hence, writing $C'_7 = 2 \log 2$, we have
\eqnsplst
{ &\E \Big[ \Reff ( \Phi(X_i) \conn \Phi(X_{i+K}) ) \,\Big|\, \cE_{(2)},\, \ell_1,\,
      \Phi(V_n) = (x,n) \Big] \\
  &\qquad \le A \left( 1 + \frac{(C'_6 + C'_7 + C'_6 C'_7) \xi}{\log \delta n} \right)
      K \delta n (\log \delta n)^{-\xi} + 1. }
Lemmas \ref{lem:old-Lemma4.3} and \ref{lem:new-Lemma7.1} give
\eqnsplst
{ &\E \Big[ \Reff ( \Phi(X_i) \conn \Phi(X_{i+K}) ) \,\Big|\, \cE_{(3)},\, \ell_1,\,
      \Phi(V_n) = (x,n) \Big] \\
  &\qquad \le A \left( 1 + \frac{C'_6 \xi}{\log \delta n} \right)
      \Big[ (\ell_1 - i \delta n) (\log (\ell_1 - i \delta n))^{-\xi} + \\
  &\qquad\qquad + ((i+1) \delta n - \ell_1) (\log ((i+1) \delta n - \ell_1))^{-\xi} \\
  &\qquad\qquad + (K - 1) \delta n (\log \delta n)^{-\xi} \Big] + 1 \\
  &\qquad \le A \left( K + \frac{C'_7 \xi}{\log \delta n} \right)
    \left( 1 + \frac{C'_6 \xi}{\log \delta n} \right) \delta n (\log \delta n)^{-\xi} + 1. }
Lemmas \ref{lem:old-Lemma4.4} and \ref{lem:new-Lemma7.1} give
\eqnsplst
{ &\E \Big[ \Reff ( \Phi(X_i) \conn \Phi(X_{i+K}) ) \,\Big|\, \cE_{(4)},\, \ell_1,\, \ell_2,\, 
      \Phi(V_n) = (x,n) \Big] \\
  &\qquad \le A \left( 1 + \frac{C'_6 \xi}{\log \delta n} \right)
      \Big[ (\ell_1 - i \delta n) (\log (\ell_1 - i \delta n))^{-\xi} + \\
  &\qquad\qquad + ((i+1) \delta n - \ell_1) (\log ((i+1) \delta n - \ell_1))^{-\xi} \\
  &\qquad\qquad + (K-2) \delta n (\log \delta n)^{-\xi} \\
  &\qquad\qquad + (\ell_2 - (i+K-1) \delta n) (\log (\ell_2 - (i+K-1) \delta n))^{-\xi} + \\
  &\qquad\qquad + ((i+K) \delta n - \ell_2) (\log ((i+K) \delta n - \ell_2))^{-\xi} \Big] + 2 \\
  &\qquad \le A \left( 1 + \frac{C'_6 \xi}{\log \delta n} \right)
     \left( K + \frac{2 C'_7 \xi}{\log \delta n} \right) \delta n (\log \delta n)^{-\xi} + 2. }
Lemmas \ref{lem:5.1} and \ref{lem:new-Lemma7.2} and the induction hypothesis give that 
for any $s = 1, \dots, K+4$ and any $s' = 1, \dots, K+4$ we have that
\eqnsplst
{ &\E \Big[ \Reff ( (z_{s'-1}, T_{s'-1}) \conn (z_{s'}, T_{s'}) ) \,\Big|\, 
     \cE^s_{(5)}, \cG_{\mathrm{tree}} \Big] \\
  &\qquad \le \begin{cases}
    1 & \text{if $s' = 2, K+3$;} \\
    A t_{s'} (\log t_{s'})^{-\xi} & \text{if $s' < s$, $s' \not= 2, K+3$;} \\
    A \left( 1 + \frac{C'_6 \xi}{\log \delta n} \right) t_{s'} (\log t_{s'})^{-\xi} 
      & \text{if $s' \ge s$, $s' \not= 2, K+3$.}
   \end{cases} }
By the triangle inequality for resistance we get that for all $s = 1, \dots, K+4$ we have
\eqnsplst
{ &\E \Big[ \Reff ( \Phi(X_i) \conn \Phi(X_{i+K}) ) \,\Big|\, 
      \cE^s_{(5)},\, \cG_{\mathrm{tree}} \Big] \\
  &\qquad \le A K \left( 1 + \frac{C'_6 \xi}{\log \delta n} \right)
      \left( 1 + \frac{2 C'_7 \xi}{\log \delta n} \right) \delta n (\log \delta n)^{-\xi} + 2. }
Hence 
\eqnsplst
{ &\E \Big[ \Reff ( \Phi(X_i) \conn \Phi(X_{i+K}) ) \,\Big|\, 
      \cE_{(5)},\, \ell_1,\, \ell_2,\, \Phi(V_n) = (x,n) \Big] \\
  &\qquad \le A K \left( 1 + \frac{C'_6 \xi}{\log \delta n} \right)
      \left( 1 + \frac{2 C'_7 \xi}{\log \delta n} \right) \delta n (\log \delta n)^{-\xi} + 2. }
Lemma \ref{lem:5.2} and the induction hypothesis (recall that 
$\cE_{(7)} = \cA(i) \cap \cB'(i,c'_0)^c$) imply that 
\eqnsplst
{ &\E \Big[ \Reff ( \Phi(X_i) \conn \Phi(X_{i+K}) ) \,\Big|\, 
      \cE_{(6)} \cup \cE_{(7)},\, \cG_{\mathrm{tree}} \Big] \\
  &\qquad \le A \left( K + \frac{2 C'_7 \xi}{\log \delta n} \right)
      \delta n (\log \delta n)^{-\xi} + 2. }
Putting the above estimates together implies that there exists $C'_8 = C'_8(K)$ such that 
\eqnsplst
{ &\E \Big[ \Reff ( \Phi(X_i) \conn \Phi(X_{i+K}) ) \,\Big|\, 
      \cA(i)^c \cup \cB'(i,c_0)^c,\, \Phi(V_n) = (x,n) \Big] \\
  &\qquad \le A \left( K + C'_8 \frac{\xi}{\log \delta n} \right)
      \delta n (\log \delta n)^{-\xi} + 2. }
With \eqref{e:i-expect-bnd} this gives
\eqnspl{e:i-expect-bnd-prob}
{ &\E \Big[ \Reff ( \Phi(X_i) \conn \Phi(X_{i+K}) ) \,\Big|\, \Phi(V_n) = (x,n) \Big] \\
  &\quad \le \frac{3 A K \delta n (\log \delta n)^{-\xi}}{4}
      \bP ( \cA(i) \cap \cB'(i,c'_0) ) \,|\, \Phi(V_n) = (x,n) ) \\
  &\quad\quad + \left[ A \left( K + C'_8 \frac{\xi}{\log \delta n} \right) 
      \delta n (\log \delta n)^{-\xi} + 2 \right] \\
  &\quad\quad\quad \times  
      \bP ( \cA(i)^c \cup \cB'(i,c'_0) \,|\, \Phi(V_n) = (x,n) ). }
Due to Theorem \ref{thm:intersect}, there exists a constant $c' = c'(K) \in (0,1)$ such that 
the expression in the right hand side of \eqref{e:i-expect-bnd-prob} is at most
\eqnsplst
{ A \delta n (\log \delta n)^{-\xi} 
    \left[ \frac{c' 3 K}{4 \log \delta n} + \left( 1 - \frac{c'}{\log \delta n} \right)
    \left( K + C'_8 \frac{\xi}{\log \delta n}  
    + \frac{2 \xi}{\log \delta n} \right) \right], }
where we bounded the remaining term $2 A^{-1} (\delta n)^{-1} (\log \delta n)^{\xi}$ by
\eqnst
{ 2 A^{-1} (\delta n)^{-1} (\log \delta n)^{\xi}
  \le \frac{2 \xi}{\log \delta n} \, \frac{n_0}{(\log n_0)^{1+\xi}} \, 
      \frac{(\log \delta n)^{1+\xi}}{\delta n}
  \le \frac{2 \xi}{\log \delta n} }
using the third inequality in \eqref{e:A-ineq} and \eqref{e:delta-n-ineq}.
We now choose $\delta_0$ and $\xi$ (depending only on $K = K_0$). In addition 
to \eqref{e:old-(7.1)} that we already required for $\delta_0$, let us also have
\eqn{e:new-(7.10)}
{ \delta_0 
  \le \frac{c'}{8 (C'_8 + 2)}, \qquad\qquad
  \big( K (2 \delta_0) + 4 \delta_0 \big) C'_6 \delta_0
  \le \frac{c'}{16} }
Let $0 < \xi \le 1/2$ satisfy:
\eqn{e:xi-ineq}
{ \xi \le \delta_0, \qquad\qquad
  \left( 1 - \frac{c'}{16 \, \log n} \right) 
  \left( 1 - \frac{\log (1/\delta_0)}{\log n} \right)^{-\xi}  
  \le 1. }
The second condition is satisfied for small $\xi$, since 
$n \ge 1/\delta_0^2$ implies that $(\log 1/\delta_0 )/\log n \le 1/2$, and hence
\eqnst
{ \left( 1 - \frac{\log 1/\delta_0}{\log n} \right)^{-\xi}
  \le 1 + 2^{1+\xi} \, \xi \, \frac{\log 1/\delta_0}{\log n}
  \le 1 + 2^{3/2} \, \xi \, \frac{\log 1/\delta_0}{\log n}. }
The first condition in \eqref{e:new-(7.10)} gives that 
\eqnspl{e:new-(7.12)}
{ &\E \Big[ \Reff ( \Phi(X_i) \conn \Phi(X_{i+K}) ) \,\Big|\, \Phi(V_n) = (x,n) \Big] \\
  &\qquad \le A K \left( 1 - \frac{c'}{8 \, \log \delta n} \right) \delta n (\log \delta n)^{-\xi} \\
  &\qquad = A n (\log \delta n)^{-\xi} K \delta \left( 1 - \frac{c'}{8 \, \log \delta n} \right) }
for $i = 0, K, \dots, (N-1)K$. For the final stretch, Lemmas \ref{lem:5.3} and
\ref{lem:new-Lemma7.1} and the induction hypothesis give that 
\eqnspl{e:new-(7.13)}
{ &\E \Big[ \Reff ( \Phi(X_{i^{\mathrm{last}}} ) \conn \Phi(V_n) \,\Big|\, \Phi(V_n) = (x,n) \Big] \\
  &\qquad \le K' A \delta n (\log \delta n)^{-\xi} \left( 1 + \frac{C'_6 \xi}{\log \delta n} \right)
      +  A \tilde{n} (\log \tilde{n})^{-\xi} \left( 1 + \frac{C'_6 \xi}{\log \delta n} \right). }
Since $K' < K$ and $\tilde{n} < 2 \delta n$, the right hand side of \eqref{e:new-(7.13)}
is at most:
\eqnspl{e:new-(7.13b)}
{ &A n (\log \delta n)^{-\xi} \left[ (K' \delta + \tilde{n}/n) 
     \left( 1 + \frac{C'_6 \xi}{\log \delta n} \right) \right] \\
  &\qquad \le A n (\log \delta n)^{-\xi} \left[ (K' \delta + \tilde{n}/n) 
      + ( K \delta + 2 \delta ) \frac{C'_6 \xi}{\log \delta n} \right]. }
Using \eqref{e:delta-ineq} and the second inequality in \eqref{e:new-(7.10)}, the
expression in \eqref{e:new-(7.13b)} is at most
\eqnst
{ A n (\log \delta n)^{-\xi} \left[ K' \delta + \tilde{n}/n 
      + \frac{c'}{16 \, \log \delta n} \right]. }
Let us sum \eqref{e:new-(7.12)} over $i = 0, K, \dots, (N-1)K$ using the triangle
inequality, and add \eqref{e:new-(7.13)}. This gives
\eqnspl{e:new-(7.14)}
{ \gamma(n,x)
  &= \sup_{m \ge 2n} \gamma_m(n,x) \\
  &\le A n (\log \delta n)^{-\xi} \Big[ N K \delta + K'\delta + \tilde{n}/n 
     - \frac{c'}{8 \, \log \delta n} 
     + \frac{c'}{16 \, \log \delta n} \Big] \\
  &= A n (\log \delta n)^{-\xi} \left( 1 - \frac{c'}{16 \, \log \delta n} \right), }
where we used that $N K \delta + K' \delta + \tilde{n}/n = 1$.

We can conclude the argument as follows. If in the definition \eqref{e:new-(7.3)} 
of $\delta$ we have $\delta_0 \le n / \| x \|^2$, then we have
\eqnsplst
{ ( \log \delta n )^{-\xi} \left( 1 - \frac{c'}{16 \, \log \delta n} \right) 
  &= (\log n)^{-\xi} \left( 1 - \frac{\log 1/\delta}{\log n} \right)^{-\xi} 
    \left( 1 - \frac{c'}{16 \, \log \delta n} \right) \\
  &\le (\log n)^{-\xi} \left( 1 - \frac{\log 1/\delta_0}{\log n} \right)^{-\xi} 
    \left( 1 - \frac{c'}{16 \, \log n} \right) \\
  &\le (\log n)^{-\xi}, }
due to the second requirement on $\xi$ in \eqref{e:xi-ineq}.
Thus the right hand side of \eqref{e:new-(7.14)} is at most $A n (\log n)^{-\xi}$.
When $\| x \| \le \sqrt{n}$, this is the claimed inequality, and when 
$\sqrt{n} < \| x \| \le \sqrt{n} / \delta_0$, it is stronger than than the 
claimed inequality.

On the other hand, if in \eqref{e:new-(7.3)} we have $\delta_0 > n / \| x \|^2$, then 
we have $\delta \ge n / \| x \|^2$ and hence $\log 1/\delta \le \log \| x \|^2/n$.
This implies that the right hand side of \eqref{e:new-(7.14)} is
\eqnsplst
{ A n ( \log \delta n )^{-\xi} \left( 1 - \frac{c'}{16 \, \log \delta n} \right) 
  &= A n (\log n)^{-\xi} \left( 1 - \frac{\log 1/\delta}{\log n} \right)^{-\xi} 
    \left( 1 - \frac{c'}{16 \, \log \delta n} \right) \\ 
  &\le A n (\log n)^{-\xi} \left( 1 - \frac{\log \| x \|^2/n}{\log n} \right)^{-\xi}, }
as claimed.
This completes the induction and the proof of Theorem \ref{thm:main}.

\end{document}